\documentclass[a4paper,reqno]{amsart}

\usepackage{a4wide}
\usepackage[english]{babel}

\usepackage{amsmath}
\usepackage{amssymb}
\usepackage{amsthm}
\usepackage[foot]{amsaddr}
\usepackage{cite}
\usepackage{centernot}
\usepackage{mathtools}
\usepackage{bm}\newcommand{\mat}[1]{\bm{\mathsf{#1}}}

\usepackage{enumerate}
\usepackage{units}
\usepackage{graphicx,color}
\usepackage{subfigure}
\usepackage{setspace}
\usepackage{wrapfig}
\usepackage{cutwin}
\usepackage[export]{adjustbox}
\usepackage{url}
\usepackage{nicefrac}
\usepackage{xcolor}

\numberwithin{equation}{section}

\newcommand{\R}{\mathbb{R}}

\newcommand{\N}{\mathbb{N}}

\newcommand{\grad}{\nabla}

\newcommand{\hoz}{\spcH^1_0(\Omega)}
\newcommand{\ho}{\spcH^1(\Omega)}
\newcommand{\hod}{\spcH^1_{\set{D}}(\Omega)}
\newcommand{\into}{\int_{\Omega}}
\newcommand{\wn}{\mathbb{W}_N}
\newcommand{\en}{e_{U}}
\newcommand{\htto}{\spcH^{2}_{\mat\beta}(\Omega)}
\newcommand{\htt}{\spcH^{2}_{\mat\beta}}
\newcommand{\T}{\mathcal{T}_h}
\newcommand{\Pb}{\Phi_{\bm{\mathsf{\beta}}}}

\newcommand{\spcL}{\mathrm{L}}
\newcommand{\spcH}{\mathrm{H}}
\newcommand{\Ltwo}{\spcL^2(\Omega)}
\newcommand{\set}[1]{\mathcal{#1}}
\newcommand{\x}{\mat{x}}
\newcommand{\dd}{\mathsf{d}}
\newcommand{\dx}{\dd\x}

\newcommand{\GD}{\Gamma_{\set{D}}}
\newcommand{\GN}{\Gamma_{\set{N}}}
\newcommand{\jN}{u_{\set{N}}}
\newcommand{\CT}{C_{\ref{eq:Trudinger}}}
\newcommand{\CS}{C_{\ref{eq:Sobolev}}}

\newcommand{\opT}{\mathsf{T}_\alpha}
\newcommand{\opTN}{\mathsf{T}_{N,\alpha}}
\newcommand{\CK}{C_{\ref{eq: quasi-optimality}}}

\newtheorem{theorem}{Theorem}[section]

\newtheorem{lemma}[theorem]{Lemma}
\newtheorem{proposition}[theorem]{Proposition}

\theoremstyle{definition}
\newtheorem{definition}[theorem]{Definition}

\newtheorem{remark}{Remark}

\begin{document}

\author{Florian Spicher \and Thomas P.~Wihler}
\address{Mathematics Institute, University of Bern, CH-3012 Switzerland}

\title[Optimal FEM for semilinear PDE with subcritical reactions]{Optimal finite element approximations of monotone semilinear elliptic PDE with subcritical nonlinearities 
}



\thanks{The authors acknowledge the financial support of the Swiss National Science Foundation (SNSF), Grant No.~$200021\underline{\phantom{a}}212868$.}

\begin{abstract}
We study iterative finite element approximations for the numerical approximation of semilinear elliptic boundary value problems with monotone nonlinear reactions of subcritical growth. The focus of our contribution is on an optimal a priori error estimate for a contractive Picard type iteration scheme on meshes that are locally refined towards possible corner singularities in polygonal domains. Our analysis involves, in particular, an elliptic regularity result in weighted Sobolev spaces and the use of the Trudinger inequality, which is instrumental in dealing with subcritically growing nonlinearities. A series of numerical experiments confirm the accuracy and efficiency of our method.
\end{abstract}

\maketitle

\section{Introduction}
    On a bounded, open, non-degenerate polygonal domain $\Omega \subset\R^2$ with a finite number of straight edges $\Gamma_1,\ldots,\Gamma_m$, we consider the semilinear boundary value problem
      \begin{alignat}{2}\label{Pb:main PDE}
            -\Delta u + g(\cdot,u) &= f&\qquad&  \text{in } \Omega\nonumber \\
            u&=0&& \text{on } \GD \\
            \partial_{\mat n} u&=0&& \text{on } \GN.\nonumber
    \end{alignat}
Here, $f\in\spcL^p(\Omega)$, for $p\in(1,\infty)$, is a given source function, and $g:\,\Omega\times\R\to\R$ represents a possibly nonlinear reaction term that is measurable with respect to its first argument, and continuously differentiable with respect to its second one, i.e., it is a Carathéodory function; we suppose that the partial derivative in the second variable, denoted by $g_u\equiv \partial_u g$, features \emph{subcritical growth}, i.e.,  for any $\sigma>0$, it holds
\begin{equation}\label{eq:gsub}
	\lim\limits_{|\xi|\rightarrow \infty}\frac{|g_u(\mat x,\xi)|}{\exp(\sigma \xi^2)}=0\qquad\text{(uniformly in $\bm{\mathsf{x}}\in\Omega$);}
\end{equation}
we also assume that $g$ is \emph{monotone} in the sense that
\begin{subequations}\label{g monotonicity}
\begin{align}
(g(\x,t_1)-g(\x,t_2))(t_1-t_2) &\ge 0\qquad \forall t_1,t_2 \in \R\quad\forall \x\in\Omega, \label{eq:gmono1}
\intertext{and that}
g(\x,t)t&\ge 0\qquad \forall t\in\R\quad\forall \x\in\Omega. \label{eq:gmono2}
\end{align}
\end{subequations}
Moreover, in order to specify the boundary data, we consider the subsets $ \GD= \bigcup_{j\in \set{D}} \Gamma_j$ and $ \GN= \bigcup_{j\in \set{N}} \Gamma_j$, which refer to Dirichlet and Neumann boundary conditions, respectively, where $\set{D}\neq\emptyset$ and $\set{N}$ are two disjunct sets with $\set{D}\cup\set{N}=\{1,\ldots,m\}$.

The focus of this paper is on \emph{optimally convergent} iterative finite element approximations of the weak formulation of~\eqref{Pb:main PDE}, which consists in finding a weak solution $u\in\hod$ such that
\begin{equation}\label{weak formulation}
        a(u,v)+b(u;v)=\into fv \;\dx\qquad \forall v \in \hod,
\end{equation}
where
\begin{equation}\label{eq:a}
      a(u,v):= \into \grad u \cdot\grad v\;\dx, \qquad u,v \in \hod,
\end{equation}
is the standard symmetric bilinear form for the Laplacian, and, for any given $w\in\hod$, the form $b(w;\cdot)$ is given by
    \begin{equation}\label{eq:b}
        b(w;v):=\into g(\x,w)v\;\dx, \qquad v \in \hod.
    \end{equation}
Throughout, we denote by $\hod$ the Sobolev space of all functions in $\spcH^1(\Omega)$ with vanishing boundary trace along $\Gamma_{\set{D}}$, equipped with the standard norm $\|\nabla(\cdot)\|_{\Ltwo}$.

The \emph{iterative} solution of monotone elliptic problems traces back to the early work of Zarantonello~\cite{Zarantonello:1960}, where the convergence of a contractive fixed point iteration scheme for strongly monotone and (globally) Lipschitz continuous operators in Hilbert spaces has been established. In recent years, this approach has gained renewed interest for the purpose of numerical solution methods for monotone elliptic boundary value problems. More specifically, the key idea is to discretize the Zarantonello scheme (more generally referred to as Picard iteration) on suitably chosen discrete Galerkin spaces (e.g. by applying finite elements), and thereby to deal with possible nonlinearities by naturally linking the iterative discretization to the underlying infinite-dimensional problem. This approach gives rise to the so-called \emph{iterative linearized Galerkin (ILG)} methodology, which has been introduced as a general abstract framework in~\cite{Heid2018}; see also \cite{CongreveWihler:2017}. Furthermore, in combination with \emph{adaptive} finite element discretizations for more specific nonlinear elliptic problems, the application of iterative linearization schemes has been developed, e.g., in~\cite{AmreinWihler:17,El-AlaouiErnVohralik:11,ErnVohralik:13}. In addition, in the articles~\cite{ghps18,ghps21,HeidPraetoriusWihler:2020}, the convergence and cost optimality of adaptive ILG discretizations for strongly monotone problems has been studied.

The fixed point technique applied in Zarantonello's original work~\cite{Zarantonello:1960} can be extended to monotone problems that are merely \emph{locally} Lipschitz continuous. This observation has been pursued in the papers~\cite{BDMS15,BDMR17,He2024} for monotone semilinear elliptic problems with monomial reactions, or involving more general nonlinearities of \emph{algebraic growth} in the works~\cite{BeckerBrunnerInnerbergerMelenkPraetorius:23,BPS:25}. Among the favorable properties of the Zarantonello iteration, which will be applied as a nonlinear solver in this paper, we point, on the one hand, to its ability to yield global convergence under mere monotonicity and subcritical growth assumptions, and, on the other hand, to the fact that the same symmetric stiffness matrix can be reused in each iteration step, thus avoiding expensive assemblies and extra regularity requirements from other nonlinear solvers; see \cite{He2024} for details. Specifically, the novelty of the present paper is twofold:
\begin{enumerate}[1.]

\item Our analysis reaches far beyond the previously mentioned works, where algebraically growing nonlinearities have been considered, and enables nonlinear reaction terms of exponential and even of subcritical growth; cf.~\eqref{eq:gsub}. Specifically, for a suitable parameter $\alpha>0$, we will prove that the Picard iteration scheme (here written in strong form) given by
\begin{equation}\label{eq:Picardstrong}
    -\Delta u_{n+1}=-\Delta u_n+\alpha \left(\Delta u_n-g(\cdot,u_n)+f\right),\qquad n\ge 0,
\end{equation}
converges to the (unique) solution of~\eqref{Pb:main PDE} as $n\to\infty$; see Thm.~\ref{thm:existence}. For this purpose, we derive a local Lipschitz bound (Lem.~\ref{lem:bbounds}) whose proof crucially hinges on the application of the \emph{Trudinger inequality}~\cite{Trudinger1967}. Our result immediately applies to discrete subspaces of $\hod$, and thus, in particular, to any conforming finite element discretization of the iteration~\eqref{eq:Picardstrong}, thereby yielding a convergent ILG method for~\eqref{Pb:main PDE} that does not require to (directly) solve any nonlinear algebraic system; see~\S\ref{sc:ILG}.

\item In addition, we develop a regularity result for~\eqref{Pb:main PDE} in a Kondratiev-type setting~\cite{kondrat1967boundary}, which takes into account possible corner singularities in the solution. More precisely, we make use of the regularity theory for linear elliptic problems developed in~\cite{Babuska1986,Babuska1979} in order to show that the solution of \eqref{Pb:main PDE} belongs to a family of corner-weighted Sobolev spaces (featuring classical $\spcH^2$-regularity in the interior of the domain $\Omega$). This, in turn, motivates the application of so-called graded mesh refinements towards the corners of $\Omega$, introduced in~\cite{Babuska1979} (see also~\cite[\S4.3]{Schwab1998}), which allow to resolve the occurring singularities at an optimal rate. Our main result, Thm.~\ref{thm:main1}, states that the proposed ILG scheme for~\eqref{Pb:main PDE} based on $\mathbb{P}_1$-FE spaces with appropriate local mesh grading converges optimally and for finitely many iterations on each discrete space.

\end{enumerate}

\subsubsection*{Outline} 
We begin by recalling the classical Trudinger inequality in \S\ref{sc:sg}, and by deriving a suitable modification that serves our purposes (Prop.~\ref{prop:Tricked T}); in particular, this allows to show the convergence of the iteration~\eqref{eq:Picardstrong} in~\S\ref{sc:ws}, and thereby to provide a constructive proof for the existence of a (unique) solution of~\eqref{Pb:main PDE} that constitutes the basis of our ILG analysis. Furthermore, we derive a regularity result in weighted Sobolev spaces in \S\ref{sc:reg}. In \S\ref{sc:optimal} we turn our attention to finite element discretizations of both the nonlinear model problem~\eqref{Pb:main PDE} and the iterative Picard method~\eqref{eq:Picardstrong} on graded meshes towards the corners of~$\Omega$ in \S\ref{sc:nonlinear}--\ref{sc:fes} and \S\ref{sc:ILG}, respectively.
Finally, we test our theoretical findings through a series of numerical experiments in \S\ref{sc:numerics}, which
highlight the practical performance of our ILG method,
and draw some conclusions in~\S\ref{sc:concl}.

%%%%%%%%%%%%%%%%%%%%%%%%%%%%%%%%%%%%%%%%%%%%%%%%%%%%%

\section{Weak Solution and Regularity}

The existence of a weak solution of~\eqref{Pb:main PDE} under the subcritical growth condition~\eqref{eq:gsub} will be established based on the theory of monotone operators, see, e.g., \cite{zeidler1990nonlinear}. In order to control the nonlinearity, we make use of the so-called Trudinger inequality.

\subsection{Subcritical growth and the Trudinger inequality}
\label{sc:sg}

For any function $\varphi:\,\Omega\times\R\to\R$ that is continuous in the second argument and has subcritical growth, i.e.,
\begin{equation}\label{eq:SCG}
\varphi(\cdot,\xi)=o(\exp(\sigma \xi^2))\qquad\text{for any $\sigma>0$ (and uniformly in $\x\in\Omega$)},
\end{equation}
we recall 
the classical Trudinger inequality (in the two-dimensional case) states that there are constants $\mu_0>0$ and $\CT>0$ (depending on $\Omega)$ such that
\begin{equation}\label{eq:Trudinger}
\sup_{\substack{v \in \spcH^1(\Omega) \\ \|v\|_{\spcH^1(\Omega)}\leq 1}} \into \exp(\mu_0 v^2) \; \dx \leq \CT;
\end{equation}
see Trudinger's original paper~\cite[Thm.~2]{Trudinger1967}.
We note that the restriction on the $\spcH^1$-norm in the supremum above is usually not satisfied for weak solutions of~\eqref{Pb:main PDE}, however, exploiting the subcritical growth property permits to circumvent this issue. To this end, for any $\varphi$ as above, we notice that the expression
\begin{equation}\label{eq:theta}
\vartheta(\varphi,\sigma):=\sup_{(\x,\xi)\in\Omega\times\R}\frac{|\varphi(\x,\xi)|}{\exp(\sigma\xi^2)}
\end{equation}
is finite for any $\sigma>0$.
Then, we establish the following result, for which we first define
 \begin{equation}\label{eq:mu}
        \mu(r):=\mu_0\min\left\{1,r^{-2}\right\},\qquad r>0,
        \end{equation}
        with $\mu_0>0$ the constant from the Trudinger inequality~\eqref{eq:Trudinger}.

    \begin{proposition}[Subcritical growth estimates]\label{prop:Tricked T}
    Consider $u\in \ho$ and $\rho>0$ such that $\|\nabla u\|_{\Ltwo}\le \rho$. If $0\le\sigma\le\mu(\rho)$, then we have the bound
        \begin{equation}\label{eq:trickT}
            \into \exp(\sigma u^2) \;\dx \leq \CT.
        \end{equation}
        Furthermore, for a function $\varphi:\,\Omega\times\R\to\R$ that has subcritical growth as in~\eqref{eq:SCG}, and any $p\in[1,\infty)$, it holds that
        \begin{equation}\label{eq:gLp}
        \|\varphi(\cdot,u)\|_{\spcL^p(\Omega)}
         \leq \CT^{\nicefrac{1}{p}}\vartheta(\varphi,\nicefrac{\mu(\rho)}{p}),
        \end{equation}
and, in particular, $\varphi(\cdot,u) \in \spcL^p(\Omega)$.
    \end{proposition}
    
    \begin{proof}
        Let $\rho>0$ and $u\in \ho$ with $\|\nabla u\|_{\Ltwo}\le \rho$. If $\rho\leq 1$, then we have $\mu(\rho)=\mu_0$, and the estimate~\eqref{eq:trickT} follows immediately by monotonicity and from the Trudinger inequality~\eqref{eq:Trudinger}:
        \[
        \into \exp\left(\sigma u^2\right)\;\dx
            \le 
             \into \exp\left(\mu_0 u^2\right)\;\dx\le \CT.
        \]      
        Otherwise, we introduce the auxiliary function $\widetilde u:= \rho^{-1}u$, and observe the bound $\|\nabla\widetilde u\|_{\Ltwo}\le 1$. Thus, employing~\eqref{eq:Trudinger}, it follows that
        \begin{equation*}
            \into \exp\left(\sigma u^2\right)\;\dx
\le            \into \exp\left(\mu_0\rho^{-2} u^2\right)\;\dx
            =
             \into \exp\left(\mu_0 \widetilde u^2\right)\;\dx\le \CT.
        \end{equation*}
Moreover, for $p\in[1,\infty)$ and a continuous function $\varphi$ with subcritical growth, applying~\eqref{eq:theta}, we conclude that
        \begin{equation*}
            \into |\varphi(\x,u)|^p \;\dx \leq \vartheta(\varphi,\nicefrac{\mu(\rho)}{p})^p  \into \exp(\mu(\rho) u^2) \;\dx \leq \CT\vartheta(\varphi,\nicefrac{\mu(\rho)}{p})^p<\infty,
        \end{equation*}
        which yields~\eqref{eq:gLp}.
    \end{proof}

\begin{remark}\label{rmk: all (SCG)-Lp}
If the partial derivative $g_u$ satisfies~\eqref{eq:gsub}, then we immediately deduce that $g$ and any of its anti-derivatives $G$ (with $\partial_u G(\x,\xi)=g(\x,\xi)$), also fulfill (SCG); in addition, all three functions are $\spcL^p$-integrable on $\hod$, for any $p\in[1,\infty)$. Indeed, by the fundamental theorem of calculus and the (SCG)-property of $g_u$, for any $\sigma\geq 0$, we infer that
        \begin{align*}
            |g(\x,\xi)| &\leq |g(\x,0)| + \int_0^\xi |g_u(\x,s)| \dd s \leq |g(\x,0)| + C \int_0^\xi e^{\sigma s^2}\;ds \\
            &\leq |g(\x,0)| + C|\xi|e^{\sigma \xi^2} \leq |g(\x,0)| + Ce^{2\sigma \xi^2},%\\
            %&\leq Ce^{\Tilde{\sigma} \xi^2},
        \end{align*}
        for all $\x\in\Omega$ and $\xi\in\R$,
        %where $\Tilde{\sigma}=2\sigma\geq 0$ is also arbitrary.
        which shows~\eqref{eq:gsub} for $g$, and analogously for $G$. The $\spcL^p$-integrability follows from the previous Prop.~\ref{prop:Tricked T}.
\end{remark}

We will now prove a few estimates for the form $b$ from~\eqref{eq:b} that will be instrumental for the subsequent convergence analysis. To this end, for any $q\in[1,\infty)$, we recall the continuous Sobolev embedding $\hod\hookrightarrow\spcL^q(\Omega)$, which is expressed in terms of the bound
    \begin{equation}\label{eq:Sobolev}
            \|v\|_{\spcL^q(\Omega)} \leq \CS(q) \|\grad v\|_{\Ltwo}\qquad\forall v\in\hod,
        \end{equation}
        for a constant $\CS(q)>0$; see, e.g., \cite[Thm.~6.3]{adams2003sobolev}.
    
\begin{lemma} \label{lem:bbounds}
Suppose that the nonlinearity $g$ in~\eqref{weak formulation} satisfies~\eqref{eq:gsub}, and let $p,p^*\in(1,\infty)$ be given such that $\nicefrac{1}{p}+\nicefrac{1}{p^*}=1$. Then, for any $u,v,w\in\hod$ and $\rho>0$ such that
\begin{equation}\label{eq:rho}
\max\{\|\nabla u\|_{\Ltwo},\|\nabla w\|_{\Ltwo}\}\le\rho,
\end{equation}
the estimates
\begin{align}
 |b(u;v)|\le \CT^{\nicefrac{1}{p}}\CS(p^*)\vartheta(g,\nicefrac{\mu(\rho)}{p})\|\grad v\|_{\Ltwo}\label{eq:b1}
 \end{align}
and
\begin{align}
|b(u;v)-b(w;v)|\le \CT^{\nicefrac{1}{p}}\CS(2p^*)^2\vartheta(g_u,\nicefrac{\mu(\rho)}{p})\|\grad(u-w)\|_{\Ltwo}\|\grad v\|_{\Ltwo}\label{eq:b2}
\end{align}
hold true, with~$\mu$ from~\eqref{eq:mu}.
\end{lemma}

\begin{proof}    
Consider $u,v\in \hod$, with $\|\nabla u\|_{\Ltwo}\le \rho$. Since $g_u$ satisfies~\eqref{eq:gsub}, we recall from Remark~\ref{rmk: all (SCG)-Lp} that $g$ is also of subcritical growth. Therefore, from~\eqref{eq:gLp}, we note that
\[
\|g(\cdot,u)\|_{\spcL^p(\Omega)}
         \leq \CT^{\nicefrac{1}{p}}\vartheta(g,\nicefrac{\mu(\rho)}{p}).
\]
Hence, applying H\"older's inequality and the Sobolev embedding~\eqref{eq:Sobolev}, for each $v\in\hod$, we establish~\eqref{eq:b1}:
    \begin{equation}\label{eq:b1a}
        |b(u;v)| \leq \|g(\cdot,u)\|_{\spcL^p(\Omega)}\|v\|_{\spcL^{p^*}(\Omega)} \leq \CT^{\nicefrac{1}{p}}\CS(p^*) \vartheta(g,\nicefrac{\mu(\rho)}{p})\|\grad v\|_{\Ltwo}.
    \end{equation} 
    In order to derive~\eqref{eq:b2}, we use the fundamental theorem of calculus to infer that
    \begin{align*}
        |b(u;v)-b(w;v)| %&\leq \left| \into (g(u)-g(w))v\;\dx\right| 
        &= \left|\into \int_0^1 \frac{\dd}{\dd s} g(\x,su+(1-s)w)v \;\dd s \;\dx\right|\\
        &= \left|\into\int_0^1 g_u(\x,su+(1-s)w) \;\dd s\; (u-w)v\;\dx\right|.
        %&\leq \into \int_0^1 |g_u(\x,su+(1-s)w)| \;\dd s \; |u-w| |v|\;\dx.
    \end{align*}
    Then, by employing a triple H\"older inequality with $q:=\nicefrac{2p}{(p-1)}=2p^*\in(2,\infty)$, so that $\nicefrac{2}{q}+\nicefrac{1}{p}=1$, we obtain
    \begin{equation*}
        |b(u;v)-b(w;v)|\leq \|u-w\|_{\spcL^q(\Omega)}\|v\|_{\spcL^q(\Omega)} \left(\into\left(\int_0^1 |g_u(\x,su+(1-s)w)| \;\dd s\right)^p \;\dx\right)^{\nicefrac{1}{p}}.
    \end{equation*}
    Since the function $t\mapsto t^p$ is convex for $p>1$, by Jensen's inequality, we deduce that
    \begin{equation}\label{eq:bdiff}
        |b(u;v)-b(w;v)|\leq \|u-w\|_{\spcL^q(\Omega)}\|v\|_{\spcL^q(\Omega)} \left(\into\int_0^1 |g_u(\x,su+(1-s)w)|^p \;\dd s \;\dx\right)^{\nicefrac{1}{p}}.
    \end{equation}
 Furthermore, in light of~\eqref{eq:gsub}, letting $\sigma:=\nicefrac{\mu(\rho)}{p}$ and recalling~\eqref{eq:theta}, we have
\[
|g_u(\x,\xi)|\leq \vartheta(g_u,\sigma)e^{\sigma\xi^2}\qquad\forall (\x,\xi)\in\Omega\times\R. 
\]
This yields
    \begin{align*}
        \int_0^1 |g_u(\x,su+(1-s)w)|^p \;ds  
        &\leq \vartheta(g_u,\sigma)^p \int_0^1 \exp(\sigma p (su+(1-s)w)^2) \;\dd s\\
        &\le \vartheta(g_u,\sigma)^p \exp(\sigma p \xi_{u,w}^2),
    \end{align*}
where we define $\xi_{u,w} := \max_{s\in[0,1]} (su+(1-s)w)$. Moreover, invoking~\eqref{eq:rho}, which implies that $\|\nabla\xi_{u,w}\|_{\Ltwo}\le\rho$, we infer from Prop.~\ref{prop:Tricked T} that
\[
\into\exp(\sigma p \xi_{u,w}^2)\;\dx=\into\exp(\mu(\rho) \xi_{u,w}^2)\;\dx\le \CT.
\]
Hence,
    employing the Sobolev inequality~\eqref{eq:Sobolev} in~\eqref{eq:bdiff}, we arrive at~\eqref{eq:b2}.
\end{proof}

\begin{remark}
    In the case of pure Dirichlet boundary conditions (i.e., when $\mathcal{N}=\emptyset$ and $\hod=\hoz$), we can refine Prop. \ref{prop:Tricked T} by invoking the so-called \emph{Moser-Trudinger inequality} established in \cite{Moser1971}, which shows that $\mu_0=4\pi$ is optimal in~\eqref{eq:Trudinger} on $\hoz$ (with the integral on the left-hand side of~\eqref{eq:Trudinger} arbitrary large for $\mu_0>4\pi$ and suitable functions $v\in\hoz$). Consequently, Lem. \ref{lem:bbounds} (and its implications) extends, in particular cases, to the situation of critical growth, i.e., when there exists $\sigma_0>0$ such that \begin{equation*}
        \lim\limits_{|\xi|\rightarrow \infty}\frac{|g(\cdot,\xi)|}{\exp\left(\sigma |\xi|^{2}\right)}=\begin{cases}
            0, & \text{ if } \sigma>\sigma_0, \\
            \infty, & \text{ if } \sigma<\sigma_0,
        \end{cases} \qquad \text{uniformly in $\x\in \Omega$.}
    \end{equation*}
    A sufficient condition for this extension is $\sigma_0 < \nicefrac{\mu(\rho)}{p} \leq \nicefrac{4\pi}{p}$.
\end{remark}

\subsection{Existence and uniqueness of weak solution}
\label{sc:ws}
 
 The existence of a weak solution to~\eqref{Pb:main PDE} will be established in a constructive way using Banach's contraction theorem. For this purpose, for a parameter $\alpha\in(0,1]$ and an arbitrary closed linear subspace $\mathbb{W}\subset \hod$ (e.g., $\hoz$ itself, or a finite-dimensional subspace), we define an operator $\opT:\,\mathbb{W}\to\mathbb{W}$ through the weak formulation
 \begin{equation}\label{eq:Tweak}
u\mapsto{\opT(u)}:\qquad a(\opT(u),v)=(1-\alpha)a(u,v)+\alpha\left(\into fv \;\dx - b(u;v)\right)\qquad\forall v\in{\mathbb{W}},
 \end{equation}
where $a(\cdot,\cdot)$ and $b(\cdot;\cdot)$ are the forms from~\eqref{eq:a} and \eqref{eq:b}, respectively,cf.~\eqref{eq:Picardstrong}.
 
For an appropriate range of $\alpha$, the ensuing result shows that $\opT$ is a well-defined, self-mapping contraction on the closed ball
\[
\set{B}_\rho:=\{v\in\hoz:\,\|\grad v\|_{\Ltwo}\le\rho\},
\]
for $\rho>0$ sufficiently large.

\begin{theorem}[Existence and uniqueness]\label{thm:existence}
Suppose that
\begin{equation}\label{eq:r}
\rho>\CS(p^*)\|f\|_{\spcL^p(\Omega)}.
\end{equation}
Then, there is $\alpha_0\in(0,1]$ such that $\mathsf{T}_{\alpha}$, with $0<\alpha\le\alpha_0$, has a unique fixed-point in~${\set{B}_\rho\cap \mathbb{W}}$, which is the (only) solution of~\eqref{weak formulation}. Furthermore, the stability bound
    \begin{equation}\label{eq:ustab}
    \|\grad u\|_{\Ltwo}\leq \CS(p^*)\|f\|_{\spcL^p(\Omega)}
    \end{equation}
    holds true.
\end{theorem}

\begin{proof}
We modify the proof of~{{\cite[Thm. 25.B]{zeidler1990nonlinear}}}, which is based on a global argument, to a local analysis on the ball ${\set{B}_\rho\cap \mathbb{W}}$. We proceed in several steps.

\begin{enumerate}[1.]

\item For any given $u\in{\set{B}_\rho\cap \mathbb{W}}$, with $\rho$ as in~\eqref{eq:r}, we begin by noticing that the right-hand side of the weak formulation \eqref{eq:Tweak} is a bounded linear functional on ${\mathbb{W}}$. Indeed, applying the Cauchy-Schwarz inequality, we see that $|a(u,v)|\le\rho\|\grad v\|_{\Ltwo}$. Moreover, from H\"older's inequality (with $\nicefrac{1}{p}+\nicefrac{1}{p^*}=1$) and the continuous Sobolev embedding~\eqref{eq:Sobolev}, we obtain
\begin{align}
\left|\into fv \;\dx\right|
&\le\|f\|_{\spcL^p(\Omega)}\|v\|_{\spcL^{p^*}(\Omega)}
\le \CS(p^*)\|f\|_{\spcL^p(\Omega)}\|\nabla v\|_{\Ltwo}.\label{eq:Cstab}
\end{align}
Finally, the term $b(u;v)$ is bounded due to~\eqref{eq:b1}:
\[
|b(u;v)|\le \CT^{\nicefrac{1}{p}}\CS(p^*)\vartheta(g,\nicefrac{\mu(\rho)}{p})\|\grad v\|_{\Ltwo}.
\] 
Hence, by the Riesz representation theorem, we conclude that $\opT(u)\in{\mathbb{W}}$ is well-defined.

\item We show that $\opT(u)\in{\set{B}_\rho\cap \mathbb{W}}$ for any $u\in{\set{B}_\rho\cap \mathbb{W}}$. For this purpose, we define $\zeta_u\in{\mathbb{W}}$ by
\begin{equation}\label{eq:zeta}
a(\zeta_u,v)=\into fv \;\dx - b(u;v)\qquad\forall v\in{\mathbb{W}}.
\end{equation}
Testing with $v=\zeta_u$, and using the estimates derived above, we infer the stability bound
\[
\|\grad\zeta_u\|_{\Ltwo}
\le L_\rho, \qquad\text{with }
L_\rho:=\CS(p^*)\left(\|f\|_{\spcL^p(\Omega)}+\CT^{\nicefrac{1}{p}}\vartheta(g,\nicefrac{\mu(\rho)}{p})\right).
\]
Furthermore, we observe the identity $\opT(u)=(1-\alpha)u+\alpha\zeta_u$, from which we deduce that
\begin{equation}\label{eq:Tid}
\begin{split}
\|\grad\opT(u)\|^2_{\Ltwo}
&=(1-\alpha)^2\|\grad u\|^2_{\Ltwo}+2\alpha(1-\alpha)a(u,\zeta_u)+\alpha^2\|\grad\zeta_u\|^2_{\Ltwo}\\
&\le(1-\alpha)^2\rho^2+2\alpha(1-\alpha)a(u,\zeta_u)+\alpha^2L^2_\rho.
\end{split}
\end{equation}
Employing~\eqref{eq:gmono2} and~\eqref{eq:Cstab}, we observe that
\[
a(u,\zeta_u)=\into fu \;\dx - b(u;u)
\le \into fu \;\dx
\le \CS(p^*)\|f\|_{\spcL^p(\Omega)}\|\nabla u\|_{\Ltwo}.
\]
Hence, we arrive at
$
\|\grad\opT(u)\|^2_{\Ltwo}
\le \Psi(\alpha),
$
where we let
\[
\Psi(\alpha):=(1-\alpha)^2\rho^2+2\alpha(1-\alpha)\CS(p^*)\|f\|_{\spcL^p(\Omega)}\rho+\alpha^2L^2_\rho.
\]
Notice that $\Psi(0)=\rho^2$ and $\Psi'(0)=-2\rho(\rho-\CS(p^*)\|f\|_{\spcL^p(\Omega)})$; the latter term is negative in view of~\eqref{eq:r}. Consequently, for any $\alpha>0$ small enough we have $0\le\Psi(\alpha)\le\rho^2$, and it follows that $\|\grad\opT(u)\|_{\Ltwo}\le\rho$.

\item Next, we establish a contraction property for $\opT$. For any $u,v\in{\set{B}_\rho\cap \mathbb{W}}$, let us recall the corresponding functions $\zeta_u,\zeta_v\in{\mathbb{W}}$ defined in~\eqref{eq:zeta}. Then, as in the previous step, we have the identity
\[
\opT(u)-\opT(v)=(1-\alpha)(u-v)+\alpha(\zeta_u-\zeta_v),
\]
which leads to
\begin{align*}
\|\grad(\opT(u)&-\opT(v))\|^2_{\Ltwo}\\
&=(1-\alpha)^2\|\grad(u-v)\|^2_{\Ltwo}+2\alpha(1-\alpha)a(\zeta_u-\zeta_v,u-v)+\alpha^2\|\grad(\zeta_u-\zeta_v)\|^2_{\Ltwo},
\end{align*}
cf.~\eqref{eq:Tid}. Using~\eqref{eq:gmono1}, it holds that
\[
a(\zeta_u-\zeta_v,u-v)=-b(u;u-v)+b(v;u-v)\le 0;
\]
furthermore, by observing that
\[
\|\grad(\zeta_u-\zeta_v)\|^2_{\Ltwo}
=a(\zeta_u-\zeta_v,\zeta_u-\zeta_v)=-b(u;\zeta_u-\zeta_v)+b(v;\zeta_u-\zeta_v),
\]
and upon applying~\eqref{eq:b2}, we obtain
\[
\|\grad(\zeta_u-\zeta_v)\|_{\Ltwo}
\le \CT^{\nicefrac{1}{p}}\CS(2p^*)^2\vartheta(g_u,\nicefrac{\mu(\rho)}{p})\|\grad(u-w)\|_{\Ltwo}.
\]
Thus,
\begin{subequations}\label{eq:Tcontract}
\begin{align}
\|\grad(\opT(u)-\opT(v))\|^2_{\Ltwo}
\le\lambda(\alpha)\|\grad(u-w)\|^2_{\Ltwo}\qquad\forall u,v\in{\set{B}_\rho\cap \mathbb{W}},
\intertext{with}
\lambda(\alpha):=(1-\alpha)^2+\alpha^2\CT^{\nicefrac{2}{p}}\CS(2p^*)^4\vartheta(g_u,\nicefrac{\mu(\rho)}{p})^2.
\end{align}
\end{subequations}
Similarly to the previous step, it can be verified that $\lambda(\alpha)<1$ for all $\alpha>0$ sufficiently small.

\item The existence and uniqueness of a fixed point $u\in{\set{B}_\rho\cap \mathbb{W}}$ of $\opT$, and thus of a solution of \eqref{weak formulation}, follows immediately from Banach's contraction theorem. Since $\rho$ can be arbitrary large, $u$ is in fact the only fixed point in ${\mathbb{W}}$. Finally, the stability estimate~\eqref{eq:ustab} {results} from taking the limit $\rho\searrow\CS(p^*)\|f\|_{\spcL^p(\Omega)}$.

\end{enumerate}
This completes the proof.
\end{proof}

 \subsection{Regularity in weighted Sobolev spaces}
\label{sc:reg}
       
In situations where the domain $\Omega \subset \mathbb{R}^2$ contains corners, as is the case for {non-degenerate} polygons, it is well-known that the inverse Laplacian $(-\Delta)^{-1}$ does not exhibit full elliptic regularity. Specifically, for right-hand side functions $f\in \Ltwo$, solutions of Poisson-type boundary value problems are typically found to be $\spcH^2$ away from the boundary~$\partial\Omega$,  however, $\spcH^2$-regularity does in general not extend uniformly to the corners. We will address this issue in terms of the corner-weighted Sobolev space $\htto$; see, e.g., \cite{Babuska1986,Babuska1979,Schwab1998}.
 To this end, we define the weight function     
    \begin{equation}\label{def:Pb}
        \Pb(\bm{\mathsf{x}})= \prod_{j=1}^m \text{dist}(\bm{\mathsf{x}},\bm{\mathsf{c}}_j)^{\beta_j},
    \end{equation}
which involves the distances from a point $\mat x\in\overline{\Omega}$ to any of the $m\geq 3$ corners $\mat c_1,\ldots,\mat c_m$ of the polygon $\Omega$, and associated exponents $\bm{\mathsf{\beta}}=(\beta_1,\ldots,\beta_m)$, with $0\leq \beta_1,\ldots,\beta_m <1$. We can then introduce the norm
     \begin{equation*}%\label{def:weighted norm}
            \|v\|^2_{\spcH^{2}_{\bm{\mathsf{\beta}}}(\Omega)} := \|v\|^2_{\spcH^{1}(\Omega)} + \sum_{\alpha_1+\alpha_2=2} {\left\|\,\Phi_{\bm{\mathsf{\beta}}} |\partial_{x_1}^{\alpha_1}\partial_{x_2}^{\alpha_2}v|\,\right\|^2_{\Ltwo}},
        \end{equation*}
          as well as the weighted Sobolev space
        \begin{equation*}
            \spcH^{2}_{\bm{\mathsf{\beta}}}(\Omega) := \big\{v \in \spcH^1(\Omega) \; : \; \|v\|^2_{\spcH^{2}_{\bm{\mathsf{\beta}}}(\Omega)} < \infty\big\}.
        \end{equation*}
We notice the continuous embedding $\htto \hookrightarrow \mathrm{C}^0(\overline{\Omega})$;
see, e.g., \cite[\S2]{Babuska1979} for a proof.

    \begin{theorem}[Regularity of weak solution]\label{thm:reg}
   Suppose that the polygon $\Omega$ is non-degenerate, i.e., all interior angles at the corners $\mat c_1,\ldots,\mat c_m$, which we signify by $\omega_1,\ldots,\omega_m$, fulfill $\omega_j\in(0,2\pi)$ for each $j=1,\ldots,m$. Furthermore, let the nonlinearity $g$ satisfy the subcritical growth condition~\eqref{eq:gsub} as well as the monotonicity conditions~\eqref{g monotonicity}, and $f\in\Ltwo$. If the weight exponents $\beta_1,\ldots,\beta_m$ in~\eqref{def:Pb} obey the bounds
        \begin{equation}\label{eq:beta LB}
           1>\beta_j > \begin{cases} 1-\min(1,\nicefrac{\pi}{\omega_j}), & \text{if $\mat c_j$ is a Dirichlet-Dirichlet or Neumann-Neumann corner}, \\ 1-\min(1,\nicefrac{\pi}{2\omega_j}), & \text{if $\mat c_j$ is a Dirichlet-Neumann corner,} \end{cases}
        \end{equation}
then  the weak solution of~\eqref{Pb:main PDE} belongs to $\hod\cap\htto$. Furthermore, the stability estimate
\begin{equation}\label{eq:regstab}
\|u\|_{\htto}\le C\left(\|f\|_{\Ltwo}+
\vartheta(g,\nicefrac{\mu(\rho)}{2})
\right)
\end{equation}
is satisfied, where $C>0$ is a constant independent of $u$, $f$, and $g$, and $\rho=\CS(2)\|f\|_{\spcL^2(\Omega)}$.
\end{theorem}

\begin{proof}
Consider the (unique) solution $u\in\hod$ of the weak formulation~\eqref{weak formulation} whose existence was established in Thm.~\ref{thm:existence}. Then, the linear functional
\[
\ell(v)=\into fv\;\dx-b(u;v),\qquad v\in\hod,
\]
with the form $v\mapsto b(u;v)$ from~\eqref{eq:b}, is bounded in $\Ltwo$; indeed, this follows immediately from~\eqref{eq:b1a} (with $p=p^*=2$) and by means of~\eqref{eq:Cstab}. Then, using the regularity shift 
\[
(-\Delta)^{-1}:\,\Ltwo\to\htto,
\]
see, e.g., \cite[Thm.~3.2]{Babuska1979} or the original work by Kondratiev~\cite[Thm.~1.1]{kondrat1967boundary}, it follows from~\eqref{Pb:main PDE} that
\[
u=(-\Delta)^{-1}\left(f-g(\cdot,u)\right)\in\htto,
\]
where the Laplacian is understood in the weak sense, i.e., in terms of the bilinear form $a$ arising in the weak formulation~\eqref{weak formulation}. Moreover, it holds the regularity estimate
\[
\|u\|_{\htto}\le C\left\|f-g(\cdot,u)\right\|_{\Ltwo}.
\]
Using~\eqref{eq:gLp} with $p=p^*=2$, we deduce that
\[
\|u\|_{\htto}\le C\left(\|f\|_{\Ltwo}+
\vartheta(g,\nicefrac{\mu(\rho)}{2})
\right),
\]
where we note that $\rho=\CS(2)\|f\|_{\spcL^2(\Omega)}\ge\|\grad u\|_{\Ltwo} $, cf.~\eqref{eq:ustab}.
    \end{proof}
       
\begin{remark}
We notice that the above regularity result and the convergence analysis in this paper can be generalized to inhomogeneous Neumann boundary conditions in~\eqref{Pb:main PDE}, i.e., for $\partial_{\mat n}u=\jN$ on $\GN$, provided that the Neumann boundary data $\jN$ belongs to an appropriate trace space; see~\cite{BabuskaGuo:1989} for details.
\end{remark}

%%%%%%%%%%%%%%%%%%%%%%%%%%%%%%%%%%%%%%%%%%%%%%%%%%%%%
       
\section{Optimal convergence of finite element approximations}
\label{sc:optimal}

We propose a numerical approximation scheme for the solution of~\eqref{Pb:main PDE} that is based on a Picard iteration scheme (also referred to as Zarantonello iteration~\cite{Zarantonello:1960}) and a suitable finite element discretization thereof. This combination gives rise to the so-called iterative linearized Galerkin (ILG) methodology that has been introduced in~\cite{CongreveWihler:2017,Heid2018}. We begin by analyzing the discretization of the weak formulation \eqref{weak formulation} in Galerkin subspaces.

\subsection{Quasi-optimal approximation in finite-dimensional Galerkin spaces}\label{sc:nonlinear}

    Let $\wn \subset \hod$ denote \textit{any} finite-dimensional (and thus closed) subspace, where the index $N$ refers to its dimension, i.e. $N=\dim(\wn)$. We begin by introducing the Galerkin discretization of the weak formulation \eqref{weak formulation} on~$\wn$: Find $U\in \wn$ such that
    \begin{equation}\label{Galerkin wn}
        a(U,v) + b(U;v) = \into fv \;\dx \qquad \forall v \in \wn.
    \end{equation}
    {By virtue of Thm.~\ref{thm:existence}}, we conclude that $U$ exists and is unique; furthermore, for $p,p^*\in(1,\infty)$ with $\nicefrac{1}{p}+\nicefrac{1}{p^*}=1$, we note that 
    \begin{equation}\label{eq:Ustab}
    \|\grad U\|_{\Ltwo}\leq \CS(p^*)\|f\|_{\spcL^p(\Omega)}, 
    \end{equation}
    cf.~\eqref{eq:ustab}. 
    
    \begin{proposition}[Quasi-optimality]
Let $\en=u-U$ signify the error between the exact solution $u\in \hod$ of \eqref{weak formulation} and its Galerkin approximation $U\in \wn$ from~\eqref{Galerkin wn}. Then,  it holds the quasi-optimality bound
    \begin{subequations}\label{eq: quasi-optimality}
    \begin{align}
        \|\grad \en\|_{\Ltwo} \leq \CK \inf_{w\in\wn}\|\grad(u-w)\|_{\Ltwo},
\intertext{with}
\CK:=1+\CT^{\nicefrac{1}{p}}\CS(2p^*)^2\vartheta(g_u,\nicefrac{\mu(\rho)}{p}),
\end{align}
\end{subequations}
and $\rho=\CS(p^*)\|f\|_{\spcL^p(\Omega)}$.
    \end{proposition}

\begin{proof} 
We first observe the Galerkin orthogonality 
    \begin{equation*}%\label{Galerkin ortho}
        a(\en,v) + b(u;v)-b(U;v) =0\qquad\forall v\in\wn.
    \end{equation*}
For any $w\in \wn$, observing that $v=w-U\in\wn$, this yields
    \begin{align*}
        \|\grad \en\|^2_{\Ltwo} 
        &= a(\en, u-w) + a(\en, w-U) 
        = a(\en, u-w) -b(u;w-U) + b(U;w-U).
    \end{align*}
Moreover, making use of~{\eqref{eq:gmono1}}, results in
\begin{align*}
\|\grad \en\|^2_{\Ltwo}
        &= a(\en,u-w)-b(u;u-U)+b(U;u-U) - b(u;w-u) + b(U;w-u)\\
        &\le a(\en,u-w) - b(u;w-u) + b(U;w-u).
    \end{align*}
Applying the Cauchy-Schwarz inequality we have that
    \begin{equation}\label{eq:aux1}
        |a(\en, u-w)|\leq \|\grad \en\|_{\Ltwo} \|\grad(u-w)\|_{\Ltwo}.
    \end{equation}
Furthermore, recalling~\eqref{eq:ustab} and~\eqref{eq:Ustab}, we notice that 
\[
\max\{\|\nabla u\|_{\Ltwo},\|\nabla U\|_{\Ltwo}\}\le \CS(p^*)\|f\|_{\spcL^p(\Omega)}=\rho.
\]
{Employing~\eqref{eq:b2}, we thus derive} the local Lipschitz bound
    \begin{equation}\label{eq:aux2}
        |b(u;w-u)-b(U;w-u)| 
        \le
        \CT^{\nicefrac{1}{p}}\CS(2p^*)^2\vartheta(g_u,\nicefrac{\mu(\rho)}{p})\|\grad\en\|_{\Ltwo}\|\grad(u-w)\|_{\Ltwo}.
    \end{equation}
    Since $w\in\wn$ is arbitrary, from~\eqref{eq:aux1} and~\eqref{eq:aux2}, we immediately obtain the estimate~\eqref{eq: quasi-optimality}.
\end{proof}

\subsection{Finite element approximations on graded meshes}
\label{sc:fes}

The approximation of functions in the weighted Sobolev space $\htto$ within a finite element setting mandates the use of suitably refined meshes that are able to properly resolve possible elliptic corner singularities. To this end, we recall the graded meshes introduced in~\cite{Babuska1979}.

    \begin{definition}[Finite spaces on graded meshes] \label{def:triangulation type}
        Let $\bm{\mathsf{\beta}}=(\beta_1,\ldots,\beta_m)\in[0,1)^m$ be a weight vector associated with the corners $\mat c_1,\ldots,\mat c_m$ of the polygon~$\Omega$, and $\Pb$ the corresponding weight function from \eqref{def:Pb}. Then, a {regular (conforming) and shape-regular triangulation} $\T=\{T\}_{T\in\T}$ of mesh size $h=\max_{T\in\T}h_T$, where $h_T$ denotes the diameter of any triangle $T\in\T$, is called a \emph{graded mesh} if there exists a constant $\kappa\ge 1$ such that, for all $T\in\T$, it holds
    \begin{align*}
            \kappa^{-1}\sup_T\Pb \leq \nicefrac{h_T}{h} \leq \kappa \inf_T\Pb,&\qquad\text{if $\Pb>0$ on $\overline{T}$},
            \intertext{and}
            \kappa^{-1} \leq \frac{h_T}{h \sup_{ T}\Pb}\leq \kappa,&\qquad\text{if there is a corner $\mat c_i$ of $\Omega$ with $\mat c_i\in\overline{T}$.}
    \end{align*}
    Furthermore, for a given triangulation $\T$, we define the associated $\mathbb{P}_1$-finite element space by
    \begin{equation}\label{eq:fes}
            \mathbb{S}^{1}_{\set{D}}(\Omega,\T) := \big\{u\in \hod \; : \; u|_{T} \in \mathbb{P}_1(T)\quad \forall T \in \T\big\},
        \end{equation}
        where $\mathbb{P}_1(T)$ denotes the set of all linear polynomials on a triangle $T\in\T$; this space consists of all continuous, element-wise linear functions on the graded mesh $\T$, with zero boundary values along~$\GD$.
    \end{definition}

The {finite-dimensional} spaces $\mathbb{S}^{1}_{\set{D}}(\Omega,\T)$, which are based on graded mesh families, are able to approximate functions in $\htto$ at an optimal rate, i.e., qualitatively comparable to the approximation of $\spcH^2$- functions on uniform meshes. More precisely, referring to~\cite[Lem.~4.5]{Babuska1979}, we have the interpolation bound
        \begin{equation*}%\label{eq:estim interp}
            \|\grad(u-I_hu)\|_{\Ltwo}\leq C h \|u\|_{\htt(\Omega)},
        \end{equation*}
where $I_h:\,\hod\to\mathbb{S}^{1}_{\set{D}}(\Omega,\T)$ is the standard nodal interpolant on~$\T$, {and $C>0$ depends on~$\Omega$, and on the triangulation parameters $\beta$ and $\kappa$ only}. Moreover, it can be seen that
\[
        N:=\dim\mathbb{S}^{1}_{\set{D}}(\Omega,\T)\lesssim h^{-2},
\]
see~\cite[Lem.~4.1]{Babuska1979}, which means that the corner refinements applied in graded meshes are sufficiently local as to preserve the order of the number of degrees of freedom occurring in uniform meshes. In particular, recalling the quasi-optimality bound~\eqref{eq: quasi-optimality} (with $p=p^*=2$) together with the regularity estimate~\eqref{eq:regstab}, we obtain the following convergence result.

\begin{theorem}[Optimal convergence of FEM]
\label{thm:main1}
Suppose that the nonlinearity $g$ from~\eqref{Pb:main PDE} satisfies the subcritical growth condition~\eqref{eq:gsub} as well as the monotonicity conditions~\eqref{g monotonicity}, and $f\in\Ltwo$. Then, on the finite element space $\wn:=\mathbb{S}^{1}_{\set{D}}(\Omega,\T)$ from~\eqref{eq:fes}, where $\T$ is a graded mesh of mesh size $h>0$ (cf.~Def.~\ref{def:triangulation type} above), the Galerkin approximation $U\in\wn$ from~\eqref{Galerkin wn} satisfies the error bound
        \begin{equation}\label{eq:qo estimate nodes}
            \|\grad(u-U)\|_{\Ltwo} \leq C N^{-\nicefrac{1}{2}} \left(1+\vartheta(g_u,\nicefrac{\mu(\rho)}{2})\right)\left(\|f\|_{\Ltwo}+
            \vartheta(g,\nicefrac{\mu(\rho)}{2})\right),
        \end{equation}
        for a constant $C>0$ {depending on $\Omega$, and on the triangulation parameters $\beta$ and $\kappa$, but} independent of $N:=\dim\mathbb{S}^{1}_{\set{D}}(\Omega,\T)$, and $\rho=\CS(2)\|f\|_{\spcL^2(\Omega)}$. 
\end{theorem}

%%%%%%%%%%%%%%%%%%%%%%%%%%%%%%%%%%%%%%%%%%%%%%%%%%%%%

\subsection{Iterative linearized finite element solution}\label{sc:ILG}

The operator $\opT$ from~\eqref{eq:Tweak} motivates a fixed point iteration 
    \begin{equation}\label{eq:fp}
    U_{n+1}=\opTN(U_n),\qquad n\ge 0,
    \end{equation}
    where $\opTN$ is the Galerkin discretization of the operator $\opT$ in a suitable finite dimensional linear subspace~$\wn$, which, for $u\in\hod$, is given by
    \[
    a(\opT(u)-\opTN(u),v)=0\qquad\forall v\in\wn.
    \]
    We can thereby approximate the nonlinear discrete formulation~\eqref{Galerkin wn} by an iterative process that was previously proposed in \cite[Eq.~(12)]{Heid2018} for strongly monotone and uniformly Lipschitz continuous problems: Specifically, for an (arbitrary) initial guess $U_0\in\wn$, the iteration~\eqref{eq:fp} generates a sequence $\{U_n\}_{n\ge0} \subset \wn$ through the following scheme:
    \begin{equation}\label{Zarantonello}
  U_{n+1}\in\wn:\qquad  a(U_{n+1}, v) = (1-\alpha)a(U_n,v) + \alpha\left(\into fv \;\dx - b(U_n;v)\right)\qquad\forall v\in\wn,
    \end{equation}
    for all $n\geq 0$ and all $v \in \hod$, cf.~\eqref{eq:Tweak}. Here, $\alpha$ is a fixed parameter satisfying $0 < \alpha \leq 1$. 
    
For the case where $\wn$ are chosen to be finite element spaces based on graded meshes, see Sec.~\ref{sc:fes}, we aim to prove that the resulting iterative approximations converge to the exact solution at an optimal rate. To this end, we first observe the following result.

    \begin{proposition}\label{prop:contraction}
        Let $\wn\subset\hod$ be a closed linear subspace. Then, for $\alpha\in(0,1]$ sufficiently small and for any initial guess $U_0\in\wn$, the iteration~\eqref{Zarantonello} converges to the (unique) solution $U\in\wn$ of~\eqref{Galerkin wn} as~$n\to\infty$. Furthermore, there is a constant $0\le r_\alpha<1$ such that the a priori error bound
        \begin{equation}\label{eq:apriori}
        \|\grad(U-U_n)\|_{\Ltwo}\le r^n_\alpha\|\grad(U-U_0)\|_{\Ltwo}\qquad \forall n\ge 0,
        \end{equation}
        holds true.
    \end{proposition} 
    
    \begin{proof}
    %We notice that the proof of Thm.~\ref{thm:existence} carries over to any closed subspace $\wn$ of $\hod$. In particular, for
    {Letting $\mathbb{W}=\mathbb{W}_N$ in Thm.~\ref{thm:existence} and} $\rho>\max\left\{\|\grad U_0\|_{\Ltwo},\CS(p^*)\|f\|_{\spcL^p(\Omega)}\right\}$, cf.~\eqref{eq:r}, the operator $\opTN$ above is a self-mapping contraction on the ball $\wn\cap\set{B}_\rho$ (for $\alpha>0$ sufficiently small), and the fixed-point iteration~\eqref{eq:fp} converges to some $U\in\wn\cap\set{B}_\rho$, which, in turn, is the unique solution of~\eqref{Galerkin wn}. The  estimate~\eqref{eq:apriori} is the classical a priori bound for contractive fixed point iterations (see, e.g., \cite[Thm.~3.7-1]{Ciarlet:2013}) with the contraction constant $r_\alpha=\lambda(\alpha)^{\nicefrac12}$ from~\eqref{eq:Tcontract}.
    \end{proof}

We are now ready to present the main result of this work which shows that the finite element approximations  generated by $n=\mathcal{O}(\log(\dim\mathbb{S}^{1}_{\set{D}}(\Omega,\T)))$ iterations (on graded meshes~$\T$), yields optimally converging discrete solutions of~\eqref{weak formulation}.

\begin{theorem}[Optimal convergence of iterative linearized FEM]\label{thm:fullconv}
Suppose that the assumptions of Thm.~\ref{thm:main1} hold, and let $\alpha>0$ be a sufficiently small parameter. Then, there is a constant $\gamma>0$ (depending on the contraction constant $r_\alpha$ from~\eqref{eq:apriori}) such that, for any initial guess $U_0\in\mathbb{S}^{1}_{\set{D}}(\Omega,\T)$, where $\{\T\}_{h>0}$ is the graded mesh family from Sec.~\ref{sc:fes}, performing $n\ge\gamma\log N$ iterations of~\eqref{Zarantonello} on $\mathbb{S}^{1}_{\set{D}}(\Omega,\T)$, it holds the a priori error bound
\[
        \|\grad(u-U_n)\|_{\Ltwo}\le CN^{-\nicefrac12},
\]
where $u\in\hod$ is the exact solution of \eqref{weak formulation}, and $N=\dim\mathbb{S}^{1}_{\set{D}}(\Omega,\T)$; the constant $C>0$ {depends on $\Omega$, and on the triangulation parameters $\beta$ and $\kappa$, but is} independent of $N$.
\end{theorem}

\begin{proof}
Applying the triangle inequality, and using the bounds~\eqref{eq:qo estimate nodes} and~\eqref{eq:apriori}, implies
\begin{align*}
        \|\grad(u-U_n)\|_{\Ltwo} \leq \|\grad(u-U)\|_{\Ltwo} + \|\grad(U-U_n)\|_{\Ltwo}\leq C\left(N^{-\nicefrac12}+r^n_\alpha\right).
\end{align*}
Then, for 
$
n\ge\nicefrac{\log N}{2|\log(r_\alpha)|},
$
we obtain $r^n_\alpha\le N^{-\nicefrac12}$,
which completes the argument.
\end{proof}

%%%%%%%%%%%%%%%%%%%%%%%%%%%%%%%%%%%%%%%%%%%%%%%%%%%%%

\section{Numerical experiments}
\label{sc:numerics}

In this section, we conduct several numerical experiments to showcase the optimal convergence rate (OCR) of the proposed iterative linearized Galerkin approach for semilinear elliptic boundary value problems~\eqref{Pb:main PDE} with (SCG)-nonlinearities, cf.~\eqref{eq:gsub}. In particular, these experiments aim to validate our theoretical results and to illustrate the performance of graded meshes to resolve possible corner singularities in polygons.
%
%The three experiments are performed using a Matlab adaptation for semilinear PDEs of the approach proposed in \cite{Alberty1999} and \cite{Funken2011}. 
We employ the Picard iteration \eqref{Zarantonello} with the initial guess $U_0\equiv 0$ being the zero function in $\hod$ and with a sufficiently small parameter $\alpha \in (0,1]$, cf.~Prop.~\ref{prop:contraction}. Throughout, we consider the L-shaped domain
\begin{equation}\label{eq:Lshaped}
\Omega=(-1,1)^2\backslash([-1,0]\times[0,1]),
\end{equation}
with a re-entrant corner at the origin. {We use a uniform initial mesh consisting of 12 triangles and 3 interior mesh nodes, and then apply successive red-green-blue (RGB) refinements on all (or a suitable subset of all) triangles to generate subdivisions that are uniform (or graded towards the re-entrant corner, respectively).} In order to compute an approximate finite element solution~$U_n$ in line with Thm.~\ref{thm:fullconv}, for a suitable value $\gamma(\alpha)\in \N$, we impose a maximal number of iterations for \eqref{Zarantonello} given by
    \begin{equation}\label{eq:gamman}
    \Tilde{n}=\gamma\lceil \log(N) \rceil.
    \end{equation}
Any integration on triangles is performed using the standard 3-point quadrature rule at the midpoints of the edges.

\subsection{Experiment 1: Exponential nonlinearity and uniform meshes}
    In this first example, we aim to test the attainability of the OCR for the exponential nonlinearity $g(u)=e^u$. Specifically, we solve the Dirichlet problem on the L-shaped domain $\Omega$ from~\eqref{eq:Lshaped}:
\begin{alignat*}{2}
            -\Delta u + e^u &= f&\qquad&  \text{in } \Omega\\
            u&=0&& \text{on } \partial\Omega.
    \end{alignat*}
    Here, we choose the right-hand side $f$ so that the exact solution is given by the smooth function $u(x,y)=\sin(\pi x) \sin(\pi y)$.
    We use a sequence of uniform meshes corresponding to $\beta=0$ in Def.~\ref{def:triangulation type}; in our plots, the associated finite element spaces are characterized by their numbers of dofs, denoted by $N$. Our focus in this experiment is on the iterative numerical scheme applied to the above problem. 
    
    Fig.~\ref{Fig:Pb1 uniform} illustrates the global error between the exact solution $u$ and the iterative solution $U_n$ against the number of degrees of freedom (dofs) $N$ for each uniform mesh. The iterative scheme \eqref{Zarantonello} is applied to each mesh until the (logarithmic) error slope between two consecutive meshes is less than $-0.49$, or the maximal number of iterations according to~\eqref{eq:gamman}
    is reached; we choose $\gamma=4$ to obtain the desired convergence rate for all the selected values of~$\alpha$, although our numerical results indicate that $\gamma=1$ would indeed suffice for larger values of $\alpha$. In the legend of Fig.~\ref{Fig:Pb1 uniform}, for each choice of~$\alpha$, we display the total number of iterations over all meshes by the variable~`$\mathrm{It}$'. The convergence plot clearly shows that the expected convergence rate of $N^{-\nicefrac{1}{2}}$ is achieved for the iterative linearized Galerkin scheme. Not surprisingly, we observe that larger $\alpha$-values lead to a more efficient performance of the Picard iteration. 

    \begin{figure}[b]
    \begin{center}
        \includegraphics[scale=0.275]{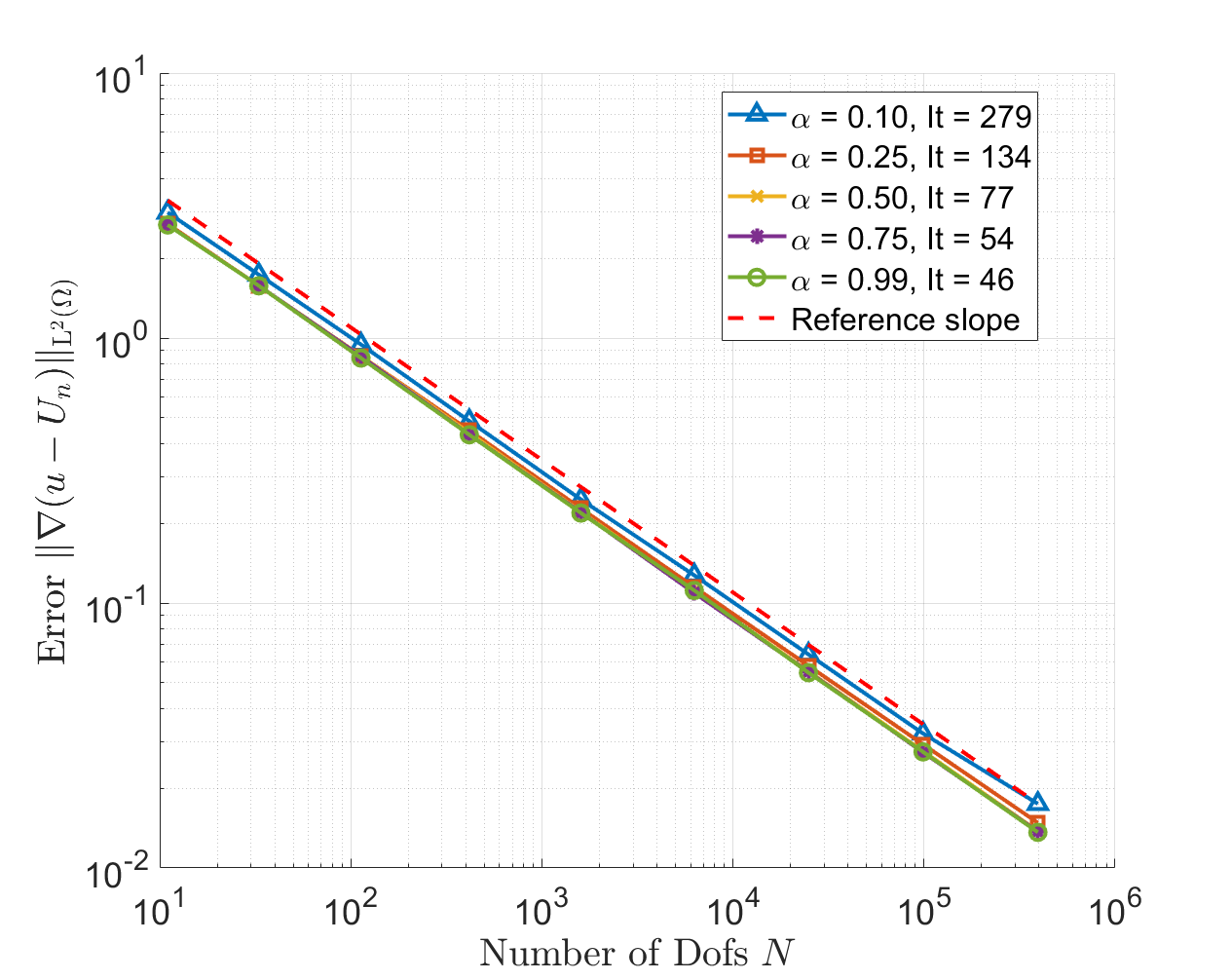}
    \end{center}
    \caption{Example~1: Convergence behavior on uniform meshes for various values of $\alpha$, and using $\gamma=4$.}
    \label{Fig:Pb1 uniform}
    \end{figure}
  
    The sensitivity with respect to the choice of $\alpha$ prompts the question of optimal parameter selection. To determine the best value of~$\alpha$ for the given example, i.e., to minimize the number of iterations to reach an error of $ \|\grad e\|_{\spcL^2(\Omega)}\leq2\cdot 10^{-2}$, we employ a golden section method on the parameter interval $[0,1]$. The plot in Fig.~\ref{Fig:Pb1 optimization} illustrate this optimization process yielding $\alpha_{\text{opt}}\approx 0.8924$ (with 2 iterations);
    here, the optimization is designed towards a minimal number of %\emph{total} 
    iterations needed to achieve the prescribed error tolerance on the last mesh (corresponding to approximately $ 3.9\cdot10^5$ degrees of freedom). 
    
    \begin{figure}
         \centering
         \includegraphics[scale=0.6]{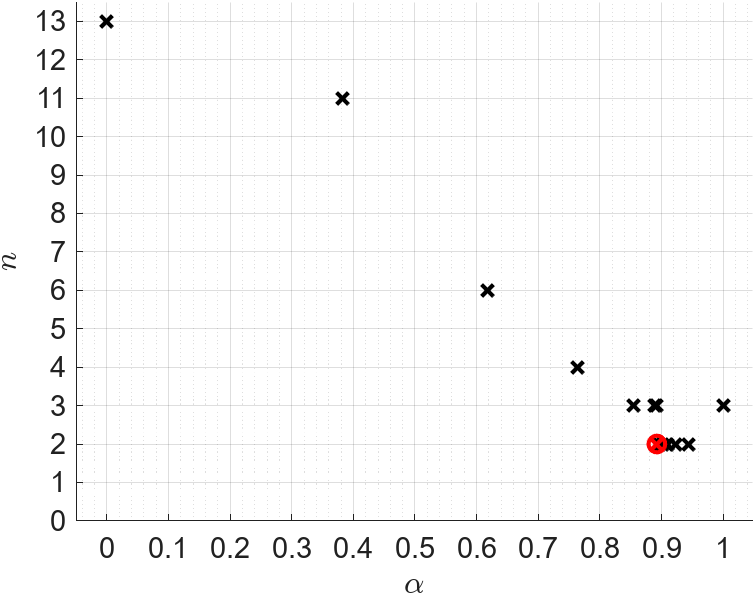}
         \caption{Experiment 1: Golden section optimization for the parameter $\alpha$ in the Picard iteration (with $\gamma=1$). Identification of $\alpha_{\text{opt}}$ (with the optimal value circled).}
         \label{Fig:Pb1 optimization}
    \end{figure}

\subsection{Experiment 2: Cubic nonlinearity on graded meshes}
    In our second experiment, we aim to demonstrate the necessity of utilizing graded meshes for the accurate approximation of a solution with a singularity at the origin in the L-shaped domain $\Omega$ from~\eqref{eq:Lshaped}. We consider the Dirichlet problem 
 \begin{alignat*}{2}
            -\Delta u + u^3 &= f&\qquad&  \text{in } \Omega\\
            u&=0&& \text{on } \partial\Omega,
    \end{alignat*}
    where $g(u)=u^3$ is a cubic monomial nonlinearity. We choose $f$ such that the exact solution is given by
    \begin{equation*}
        u(x,y)=2r^{-\nicefrac{4}{3}}xy(1-x^2)(1-y^2),
    \end{equation*}
    with the radial variable $r=(x^2+y^2)^{\nicefrac12}$, which has a singularity at the origin~$\mat{c}_1=(0,0)$; indeed, it can be verified that $u\in\spcH^2_{\beta}(\Omega)$, for the weight function $\Phi_{\mat\beta}=r^\beta$, with $\beta>\nicefrac13$, which is in line with the pure Dirichlet case in~\eqref{eq:beta LB}. The re-entrant corner singularity necessitates the use of graded mesh refinements towards $(0,0)$, where we use $\beta=0.4$. %The corresponding graded meshes are constructed based on appropriate red-green-blue (RGB) refinements of triangles close to~$\mat c_1$; 
    The corresponding graded meshes are constructed based on appropriate {RGB refinements of triangles close to~$\mat c_1$, with values $h\in\{0.25,0.15,0.08,0.035,0.016,0.008,0.0038,0.0019\}$;} we present two examples in Fig.~\ref{Fig:Pb2 meshes}.

    \begin{figure}
         \centering
        \includegraphics[scale=0.56]{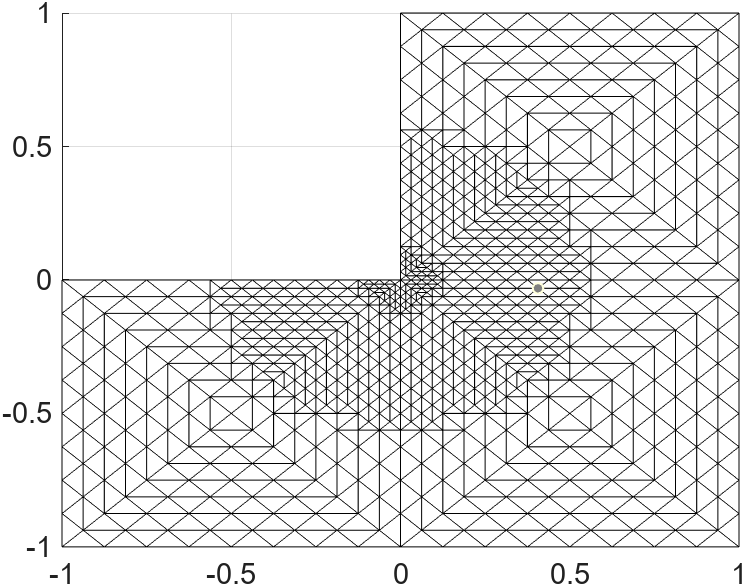}\hfill
         \includegraphics[scale=0.56]{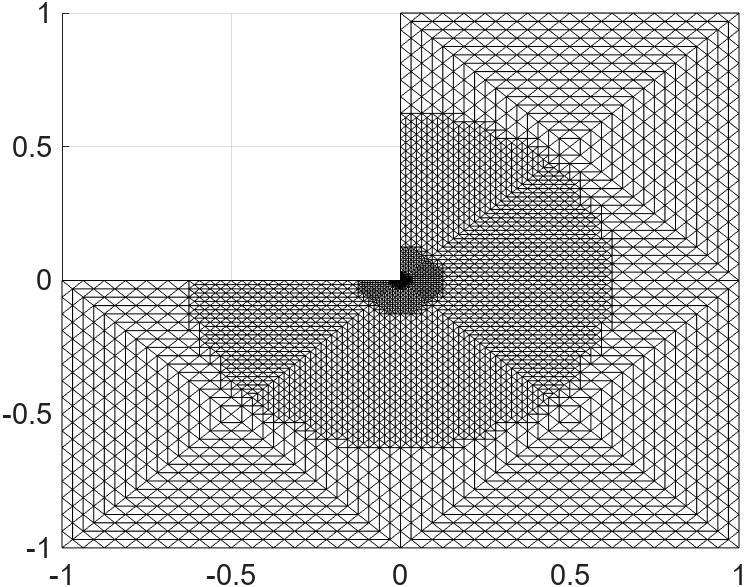}
         \caption{Graded meshes towards the origin in an L-shaped domain with $\beta=0.40$ for mesh size $h=0.035$ (left) and $h=0.016$ (right).}
         \label{Fig:Pb2 meshes}
    \end{figure}
    
    Fig.~\ref{Fig:Pb2 rates} compares the convergence rates on graded (left) and on uniform meshes (right). Whilst the OCR is reached on graded meshes, we clearly see a suboptimal convergence regime for uniform meshes (utilizing the maximal number of iterations at hand for each $\alpha$). This is because uniform meshes do not effectively distribute the available degrees of freedom in the presence of the occurring corner singularity.
    
    \begin{figure}
         \centering
         \includegraphics[scale=0.37]{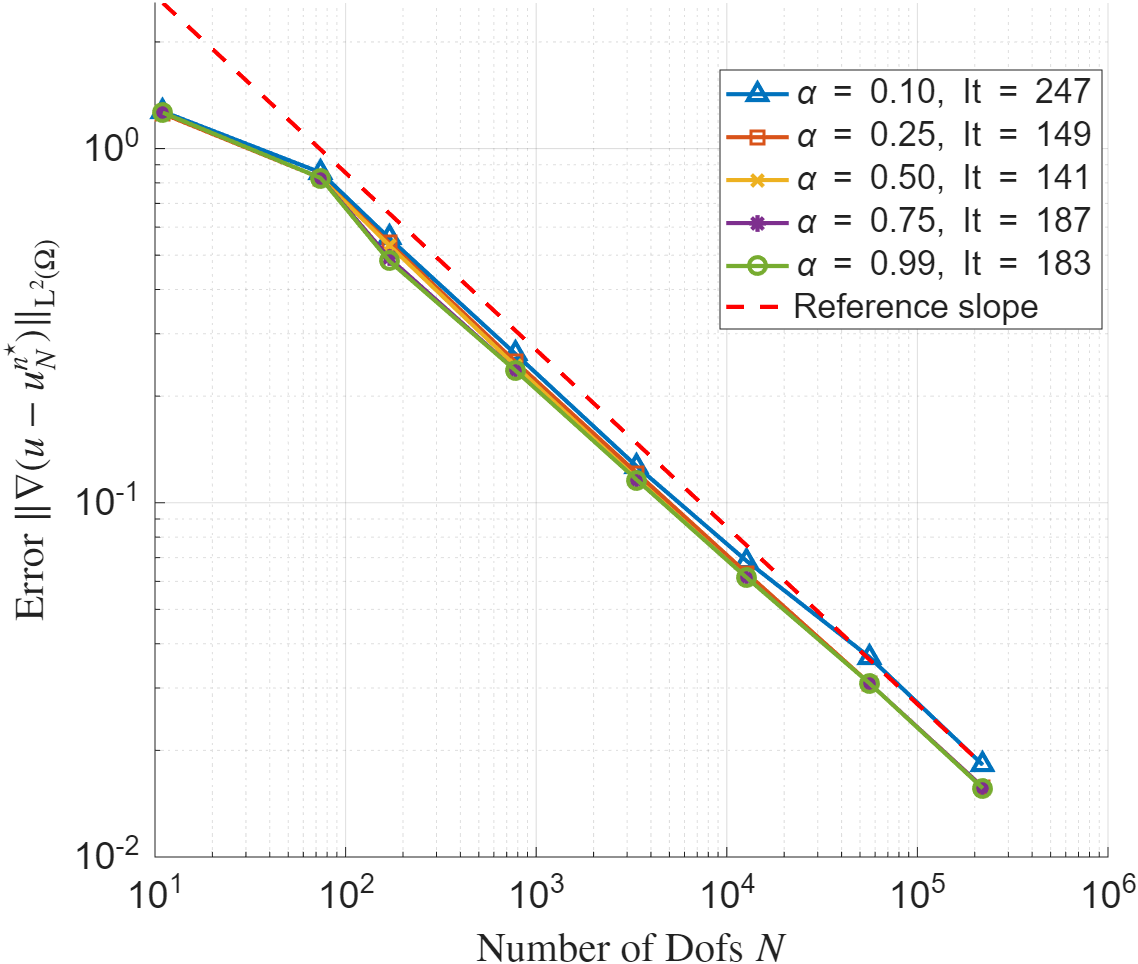}\hfill
         \includegraphics[scale=0.345]{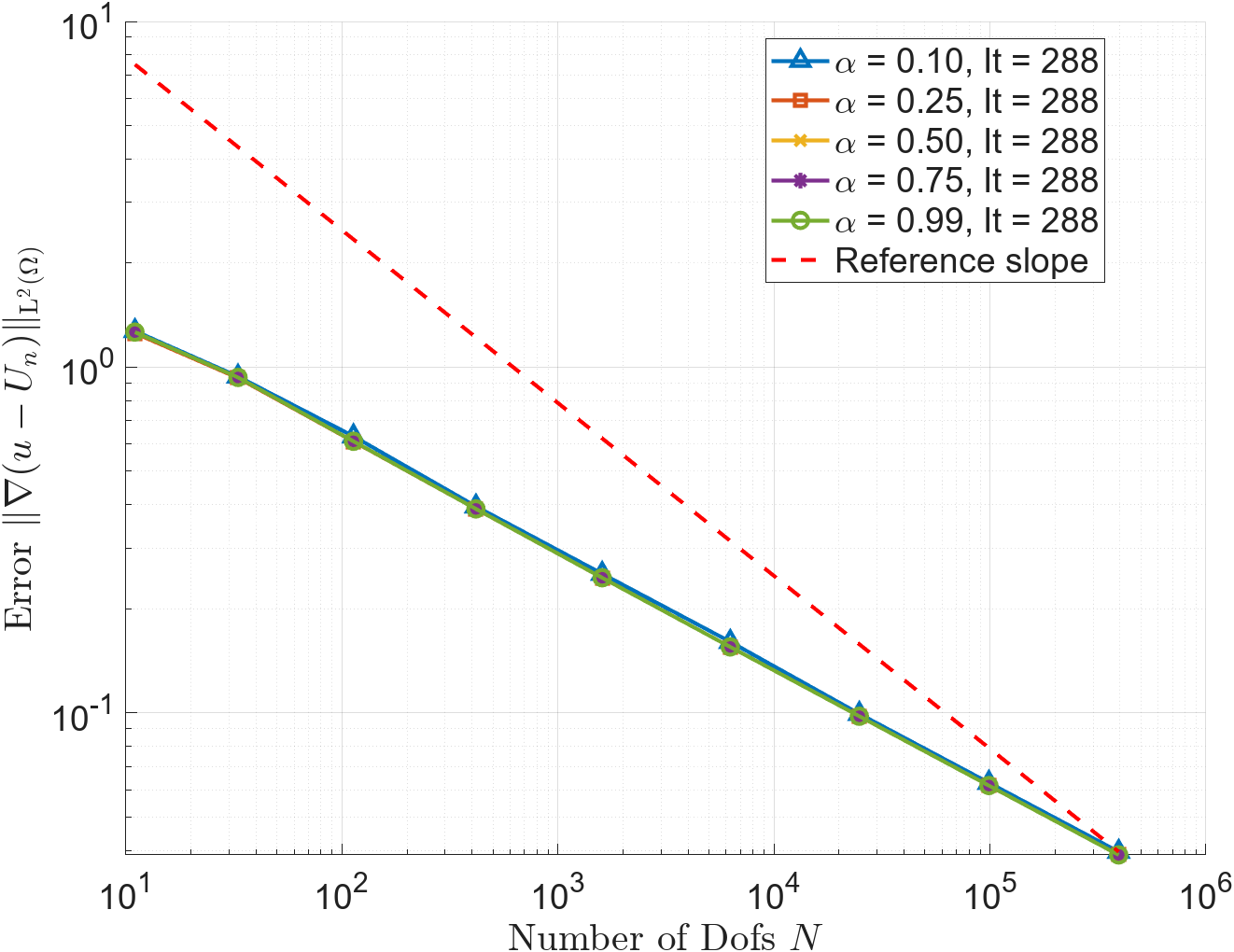}
         \caption{Experiment 2: Computational performance for several values of $\alpha$ and their corresponding total number of iterations used to reach the OCR with $\gamma=4$ on graded meshes (left), and uniform meshes (right); the dashed line represents a reference slope of $-\nicefrac{1}{2}$.}
         \label{Fig:Pb2 rates}
    \end{figure}

    Similarly to the first experiment, we perform an optimization analysis in the parameter $\alpha$. Fig.~\ref{Fig:Pb2 optimization} depicts the golden section method for $\gamma=1$. More precisely, we illustrate the optimization process to reach the prescribed error of $\|\grad e\|_{\spcL^2(\Omega)}\leq 2\cdot 10^{-2}$ on a graded mesh with roughly $3.9\cdot 10^5$ degrees of freedom. We require 2 iterations with a corresponding optimal value of $\alpha_{\text{opt}}\approx 0.9152$, thereby underlining the effectiveness of the graded mesh approach in handling singularities at corners.
    \begin{figure}
         \centering
         \includegraphics[scale=0.6]{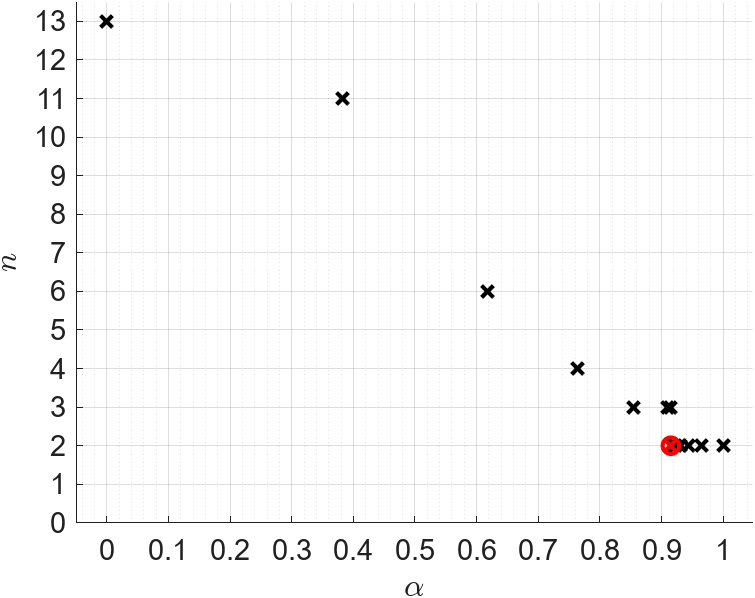}
         
         \caption{Experiment 2: Golden section optimization for the parameter $\alpha$ in the Picard iteration (with $\gamma=1$). Identification of $\alpha_{\text{opt}}$ (with the optimal value circled).}
         \label{Fig:Pb2 optimization}
    \end{figure}

\subsection{Experiment 3: Exponential nonlinearity and mixed boundary data}

    In this final example, we aim to test the effectiveness of graded meshes on a model problem with mixed boundary conditions and a (near) limit case of exponential nonlinearity in the sense of Def.~\ref{eq:SCG}. Specifically, on the L-shaped domain $\Omega$ from~\eqref{eq:Lshaped}, we consider the boundary value problem 
    \begin{alignat*}{2}
            -\Delta u + e^{4|u|^{0.9}u} &= f&\qquad&  \text{in } \Omega \\
            u&=0&& \text{on } \GD \\
            \partial_{\mat n} u&=0&& \text{on } \GN,
    \end{alignat*}
    where $\Gamma_{\set{N}}=\{0\}\times(0,1)$ and $\Gamma_{\set{D}}=\partial\Omega \backslash \overline{\Gamma}_{\set{N}}$. 
    We choose $f$ such that the exact solution is given by
    \begin{equation*}
        u(x,y)=r^{-\nicefrac{2}{3}}y(1-x^2)(1-y^2),
    \end{equation*}
    with a singularity at the corner $\bm{\mathsf{c}}_1=(0,0)$; it holds $u\in\spcH^2_{\beta}(\Omega)$ for the weight function $\Phi_\beta=r^\beta$, with $\beta>\nicefrac23$, which corresponds to the mixed Dirichlet-Neumann case in~\eqref{eq:beta LB}. Accordingly, we propose to use graded meshes based on $\beta=0.7$,
    for reasonable iteration parameters $\alpha=0.5$ and $\gamma=2$, {and with the $h$-values chosen exactly as in Experiment 2.}
    From Fig.~\ref{Fig:Pb3 meshes} we observe that the desired optimal rate is attained quickly, thereby demonstrating the benefit of the adjusted parameter settings in the current situation of mixed boundary conditions.

    \begin{figure}
         \centering
         \includegraphics[scale=0.368]{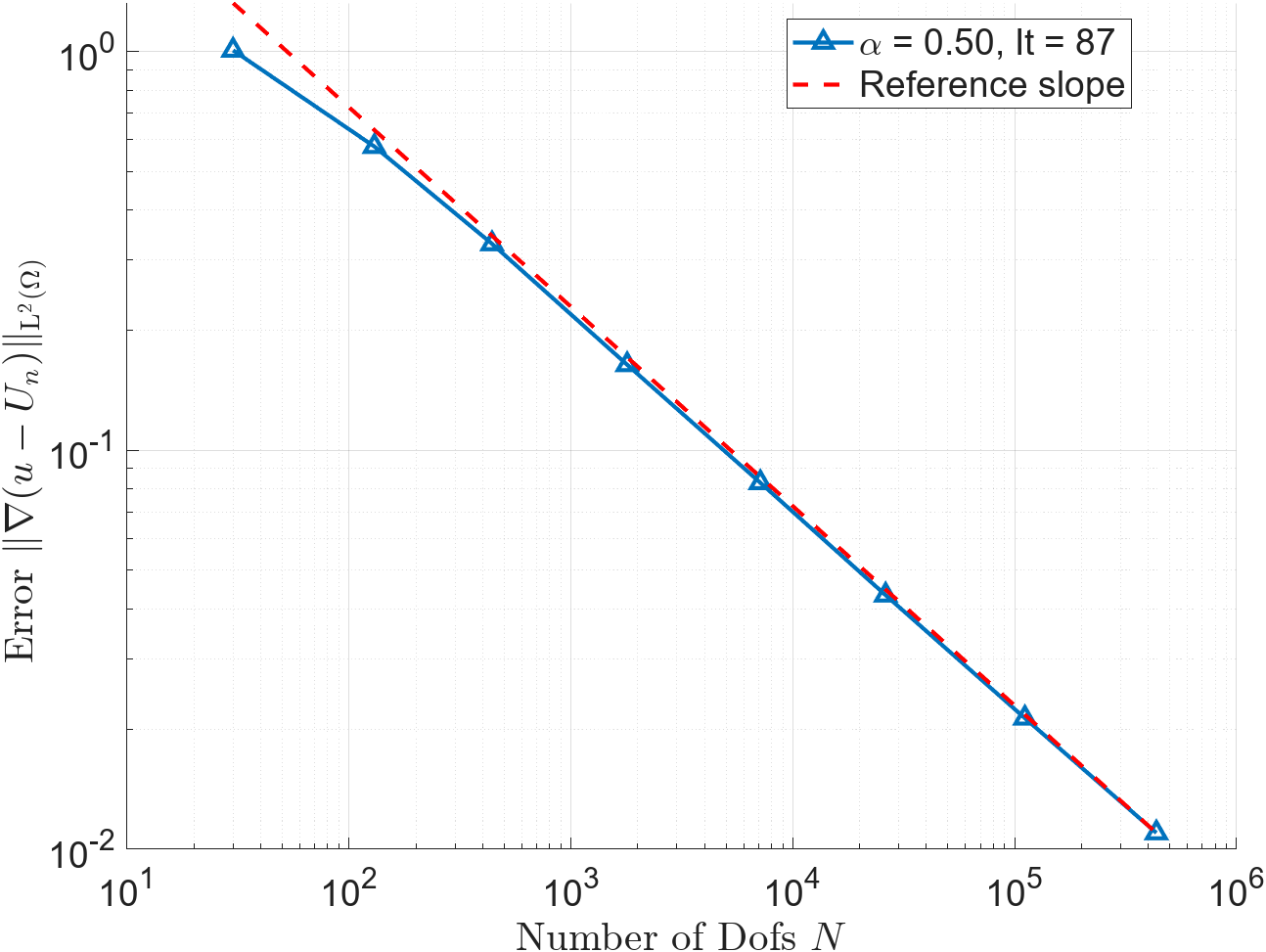}
         \caption{Experiment 3: Convergence behavior for a sequence of graded meshes with $\alpha=0.5$, $\beta=0.7$ and $\gamma=2$.}
         \label{Fig:Pb3 meshes}
    \end{figure}
    
    Furthermore, as in the previous experiments, Fig.~\ref{Fig:Pb3 optimization} depicts a computational analysis of the Picard scheme for various values of the iteration parameter~$\alpha$ to reach the prescribed error of $\|\grad e\|_{\spcL^2(\Omega)}\leq 2\cdot 10^{-2}$  on a mesh of approximately $4.3\cdot10^5$ degrees of freedom; for $\alpha_{\text{opt}}\approx 0.7416$, which results in 3 iterations, we perceive the high efficiency of the graded meshes to handle singularities, even with mixed boundary conditions. However, in contrast to Fig.~\ref{Fig:Pb1 optimization} and \ref{Fig:Pb2 optimization}, we observe {in Fig.~\ref{Fig:Pb3 optimization}} a more unstable behavior in the required number of iterations, as well as a potential divergence of the scheme for larger values of $\alpha$. 
    In fact, this is made even clearer by Fig.~\ref{Fig:Pb3 div iterations}, where we increase the scaling parameter for the number of iterations~$\Tilde{n}$, cf.~\eqref{eq:gamman}, from $\gamma=1$ to $\gamma=10$ to reach the OCR {on the entire mesh sequence}. We recognize three different stages in the behavior with respect to~$\alpha$: For intermediate values up to approximately $\alpha=0.79$, a stable, oscillatory behavior is observed; beyond this range, we enter a small transitional phase where the number of iterations increases considerably; finally, from around $\alpha=0.91$ onward, the Picard scheme does no longer converge within a sensible number of iterations.

    \begin{figure}
         \centering
         \includegraphics[scale=0.6]{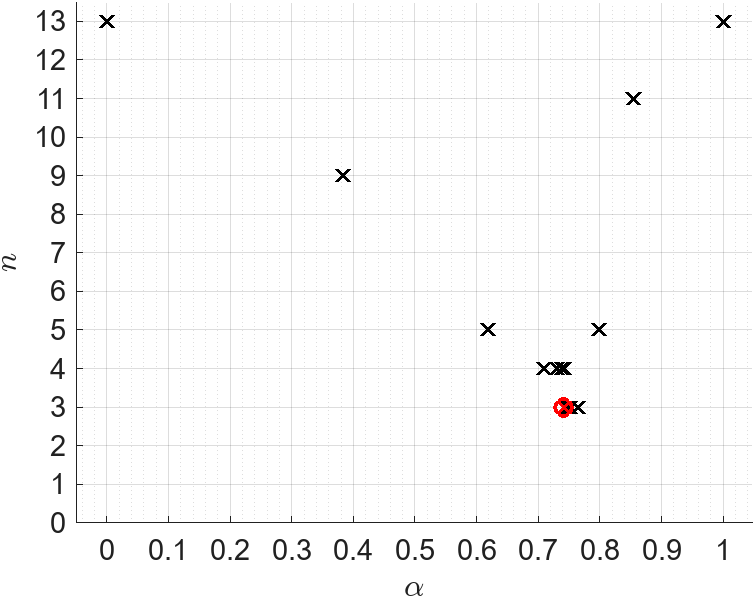}
         \caption{Experiment 3: Golden section optimization for the parameter $\alpha$ in the Picard iteration (with $\gamma=1$). Identification of $\alpha_{\text{opt}}$ (with the optimal value circled).}
         \label{Fig:Pb3 optimization}
    \end{figure}

    \begin{figure}
         \centering
         \includegraphics[scale=0.36]{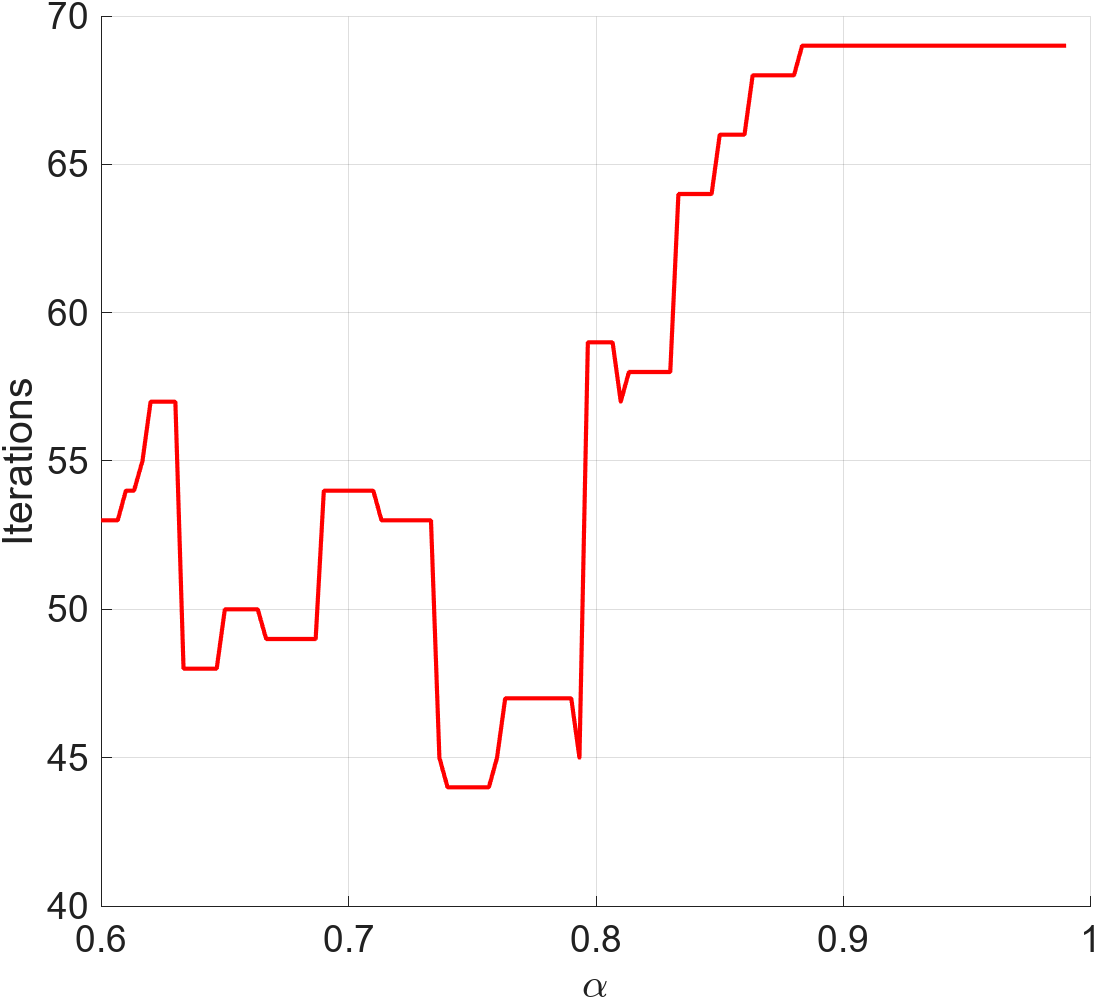}\hfill
         \includegraphics[scale=0.36]{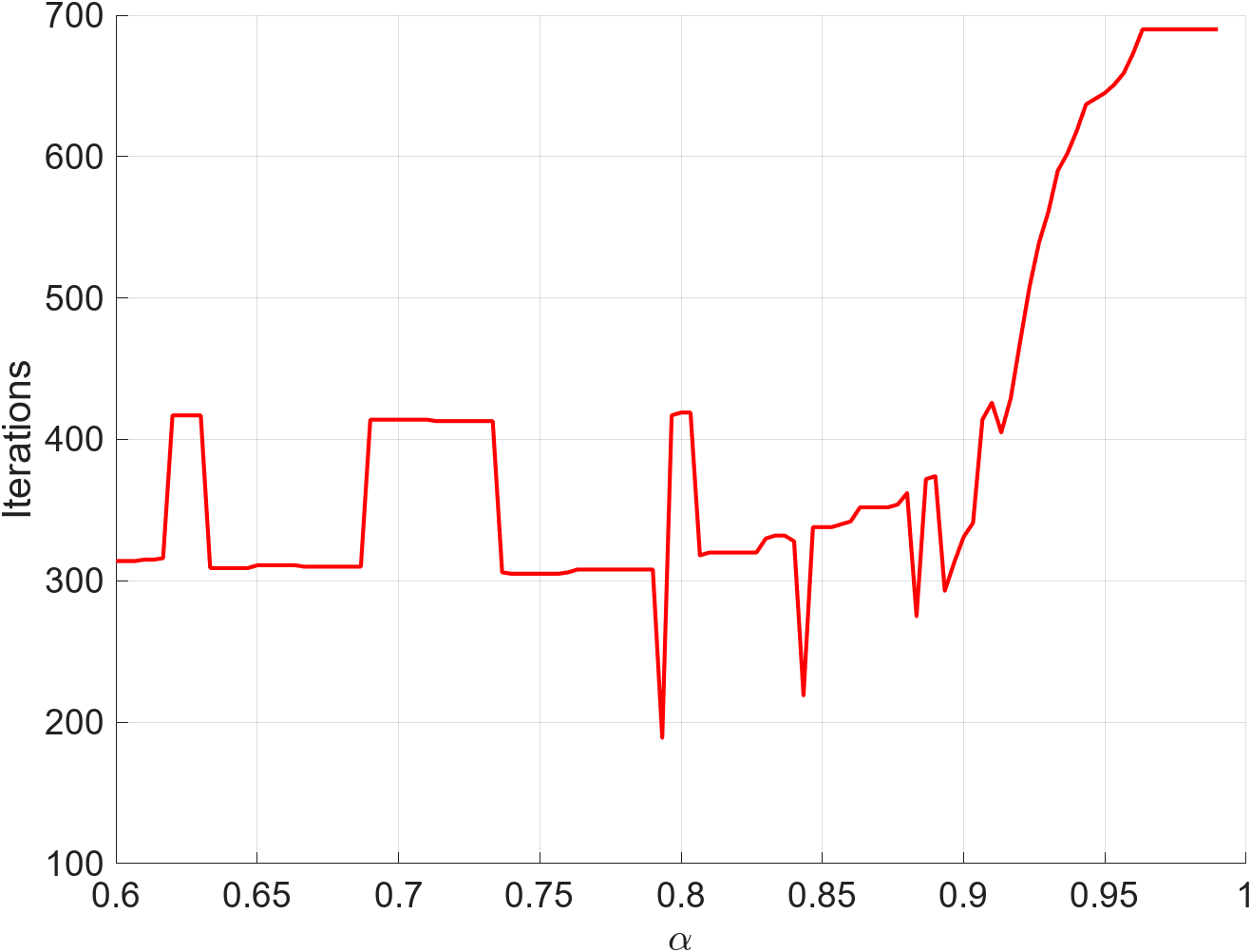}
         \caption{Experiment 3: Numbers of iterations for $\gamma=1$ (left) and $\gamma=10$ (right) needed {across all meshes} in terms of $\alpha\in [0.6,0.99]$ for the nonlinearity $g(u)=\exp(4|u|^{0.9}u)$ and mixed boundary conditions. Distance between two $\alpha$-samples: $\Delta\alpha=\nicefrac{1}{30}$.}
         \label{Fig:Pb3 div iterations}
    \end{figure}
    
    We further investigate the failure of the Picard iteration to converge for values~$\alpha\approx 1$. To this end, we employ two preset maximal iteration limits resulting from $\gamma=1$ (i.e. ${\sum} \Tilde{n}=69$) and $\gamma=10$ (i.e. ${\sum}\Tilde{n}=690$). Fig.~\ref{Fig:Pb3 div stages} shows three distinct behaviors for both $\gamma=1$ and $\gamma=10$: For $\gamma=1$, we see that convergence begins to deteriorate from~$\alpha=0.85$ for larger numbers of degrees of freedom, yet, remains stable for $\gamma=10$. 
    For $\alpha=0.95$, however, neither value of~$\gamma$ is able to retain convergence.
    In conclusion, for higher~$\alpha$-values (uniformly away from a certain upper limit $<1$), OCR can still be retrieved by appropriately increasing the number of iterations. 
    For $\alpha$-values close to~1, however, the iteration seems to diverge completely, thereby indicating that these values may possibly exceed the threshold presented in Prop.~\ref{prop:contraction}.

    \begin{figure}
         \centering
         \includegraphics[scale=0.367]{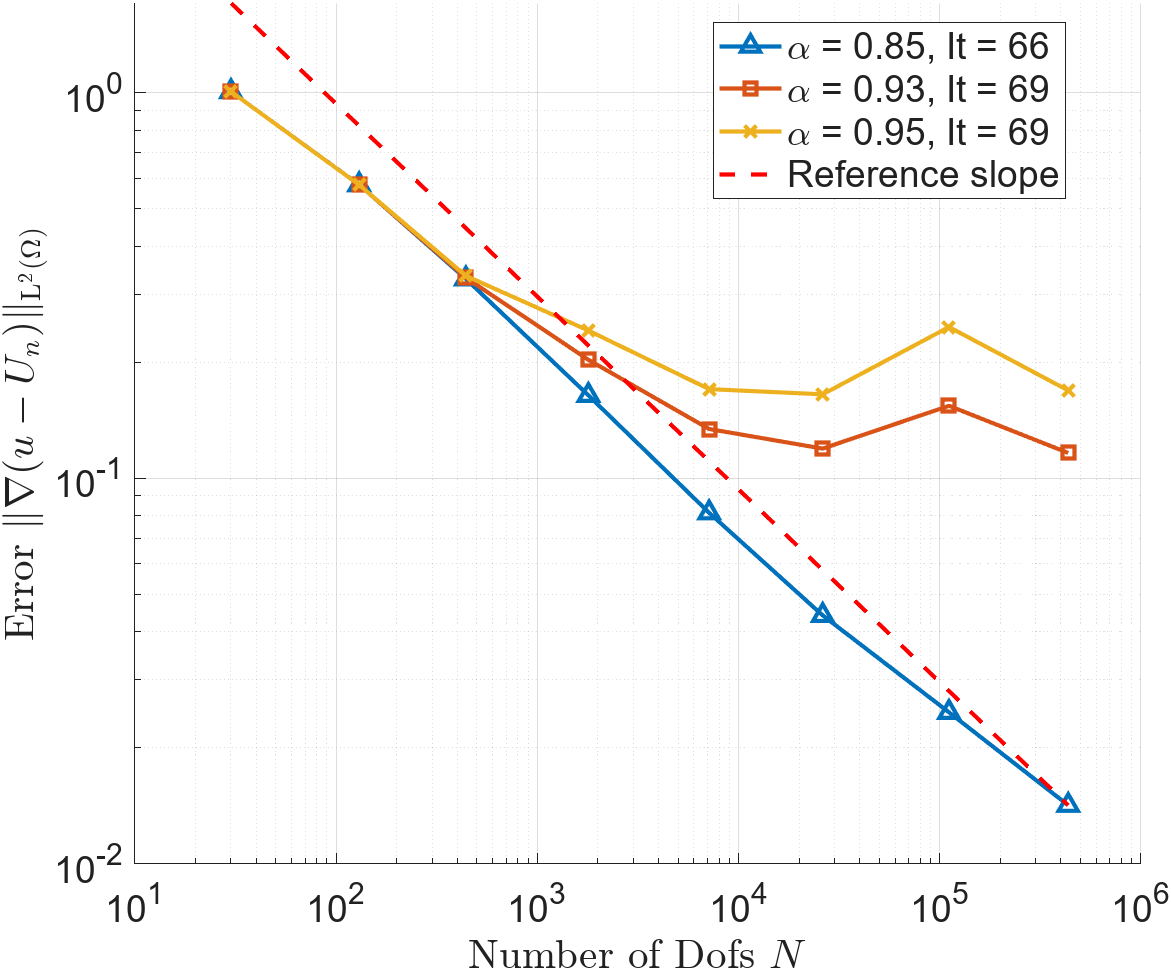}\hfill
         \includegraphics[scale=0.367]{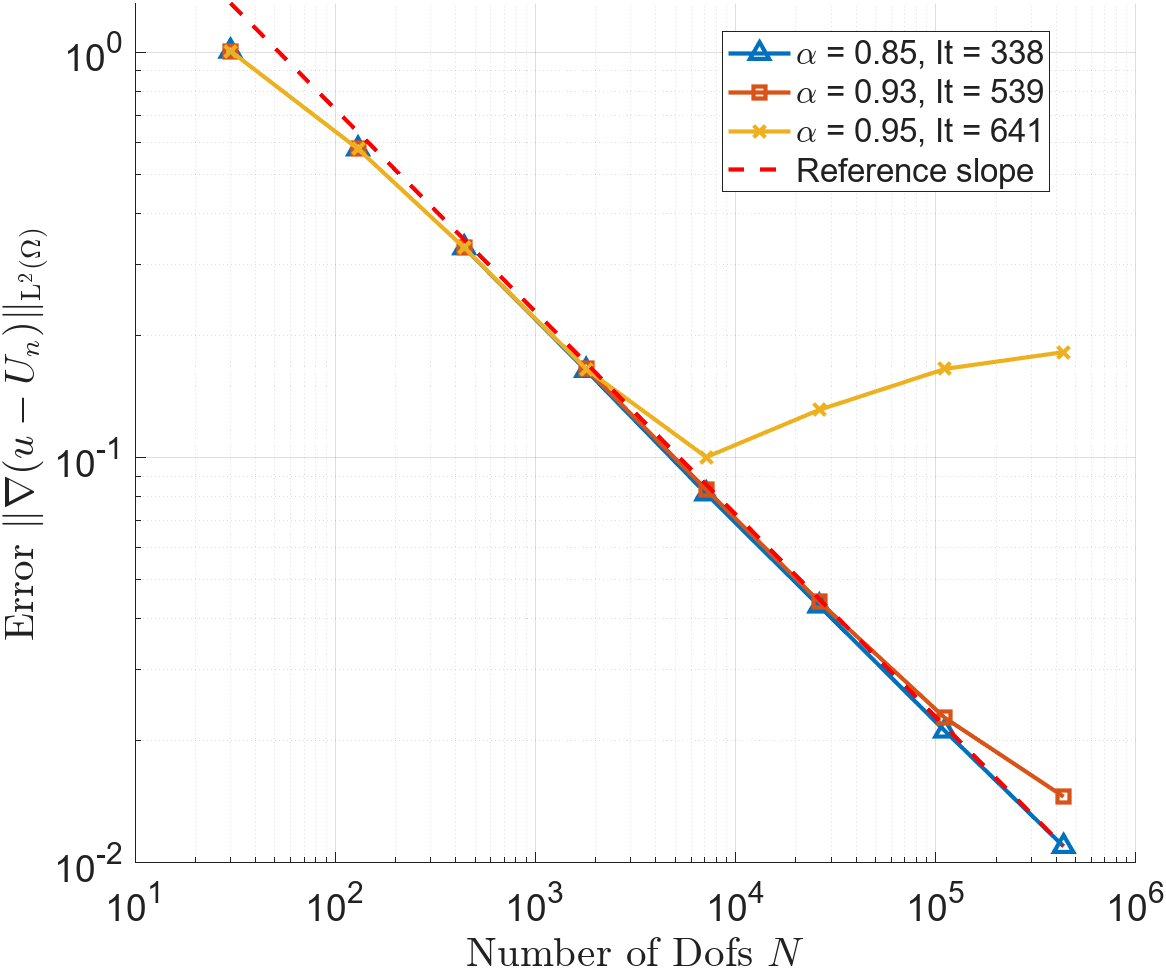}
         \caption{Experiment 3: Stages of divergence for various $\alpha$ values with $\gamma=1$ (left) and $\gamma=10$ (right).}
         \label{Fig:Pb3 div stages}
    \end{figure}

        In summary, our numerical study
        across the three experiments suggests that $\alpha \simeq 0.8$ is generally a good choice. 
        This observation is reinforced by Fig.~\ref{Fig:Pb3 div iterations}, which shows that $\alpha \simeq 0.8$ remains effective given a sufficient number of iterations is being performed. 
        This supports the application of the Picard iteration as an iterative nonlinear solver for the problem class under consideration. In particular, this method is stable and robust for a wide range of $\alpha$ values, requiring only a few iterations to reach optimal convergence rates, even for strong exponential nonlinearities.

\subsection{Additional discussion: Scheme stability}
    Finally, in Figs.~\ref{Fig:Pb3 no mixed}--\ref{Fig:Pb3 only mixed}, we explore how the strength of the nonlinearity $g(u)$ and the type of boundary conditions influence the convergence behavior of the Picard iteration.  We fix $\gamma=10$ in~\eqref{eq:gamman}, and consider the range $\alpha\in [0.6,0.99]$ for uniformly spaced samples of distance $\Delta\alpha=\nicefrac{1}{30}$ as before. 

     \begin{figure}
         \centering
         \includegraphics[scale=0.35]{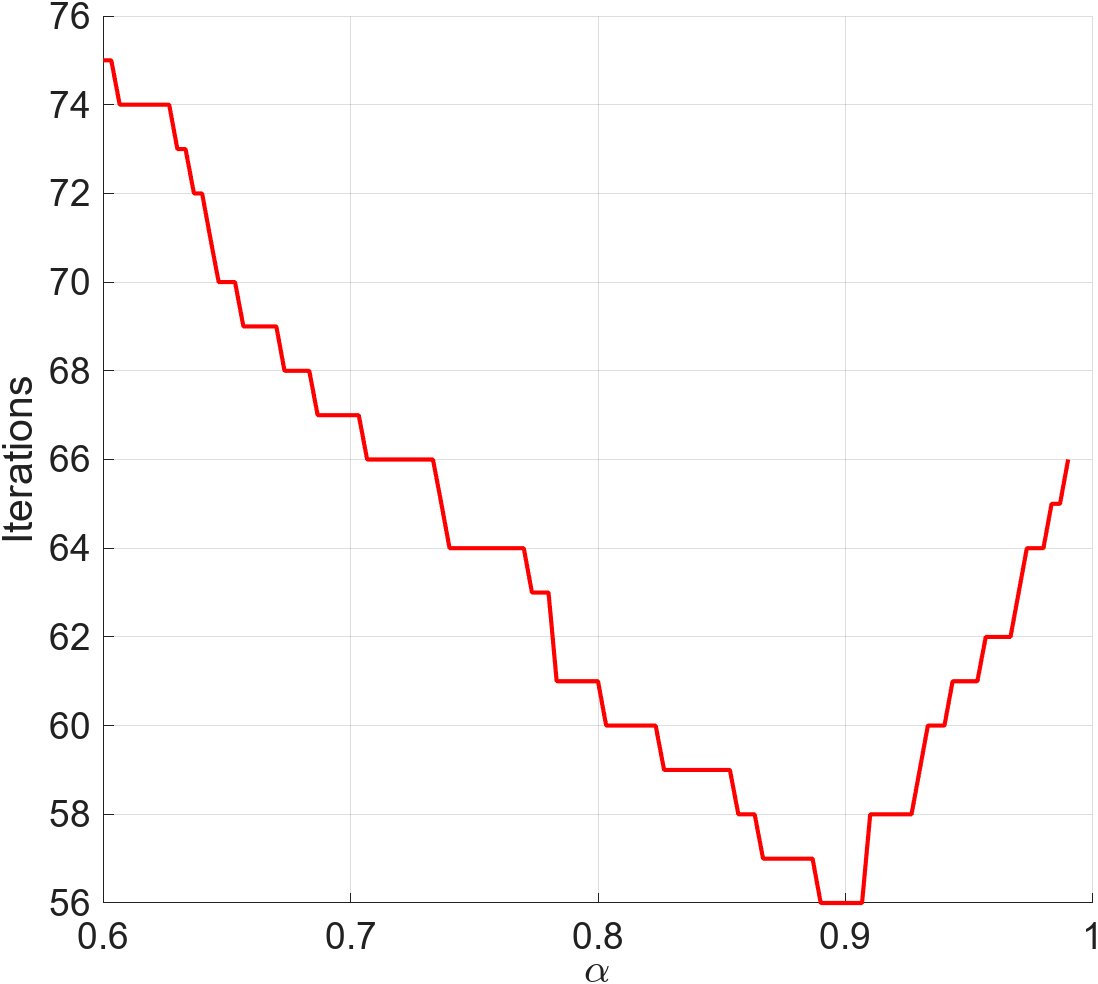}\hfill\includegraphics[scale=0.345]{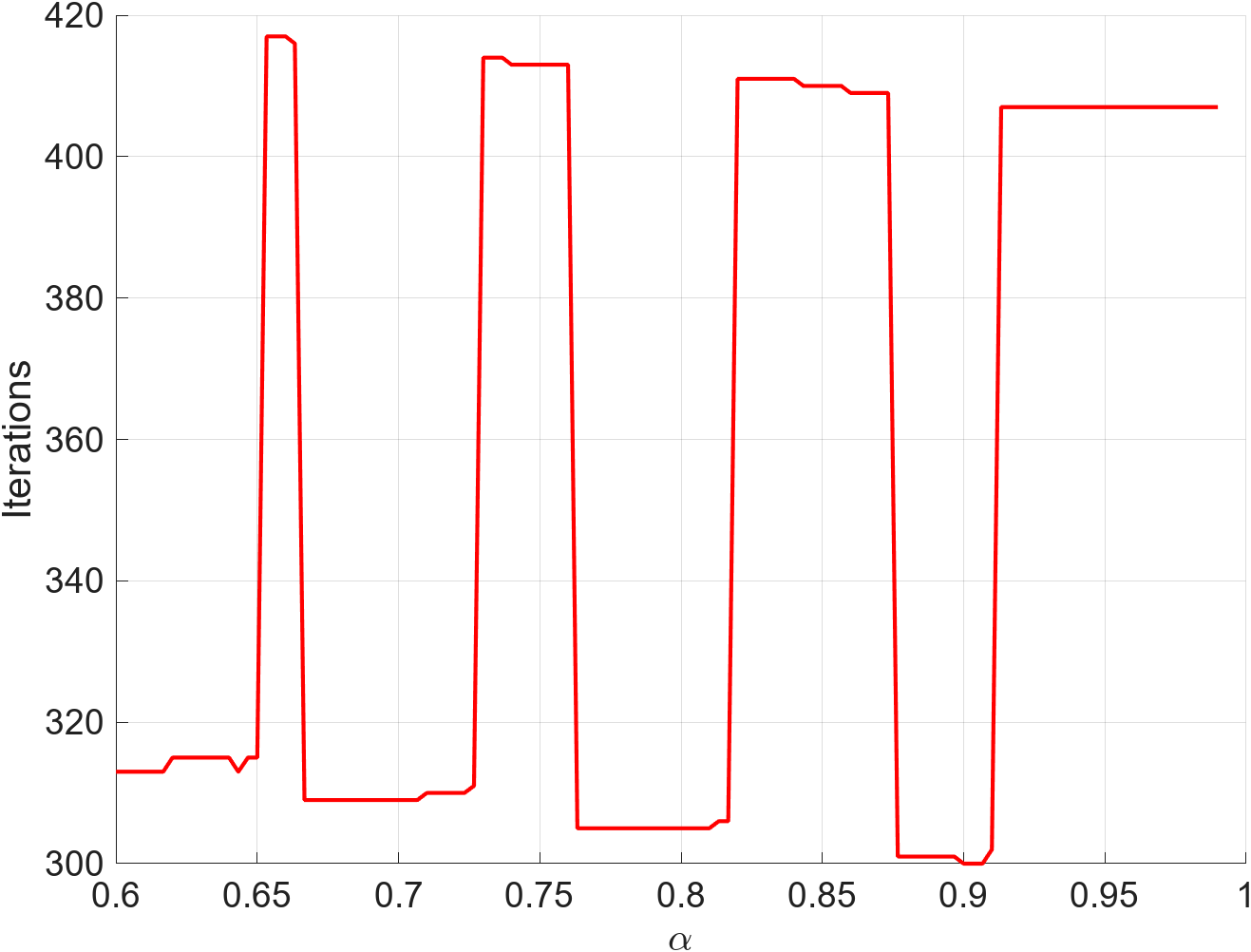}
         \caption{Numbers of iterations needed {across all meshes} in terms of $\alpha$ for $\gamma=10$. Left: Nonlinearity $g(u)=\exp(4|u|^{0.9}u)$ and \emph{Dirichlet boundary conditions}. Right: Nonlinearity $g(u)=\exp(|u|^{0.9}u)$ and \emph{mixed boundary conditions}.}
         \label{Fig:Pb3 no mixed}
    \end{figure}

    In Fig.~\ref{Fig:Pb3 no mixed} we compare {throughout the mesh sequence} homogeneous Dirichlet and mixed boundary data for nearly critical nonlinearities. We see that the type of boundary conditions significantly impacts the convergence regime of the Picard iteration. In the former case (left plot), in spite of a larger constant in the exponent of the nonlinearity, a parabolic shape that resembles the plots in Figs.~\ref{Fig:Pb1 optimization}, \ref{Fig:Pb2 optimization} and~\ref{Fig:Pb3 optimization} is obtained. However, compared to the previous experiments, the iteration numbers are considerably larger. This is even more pronounced in the right plot, where mixed boundary conditions are investigated; in addition, step-like plateaus appear, where the number of iterations are locally constant for increasing values of $\alpha$. Evidently, this oscillatory behavior is quite unstable, as slight variations in $\alpha$ can cause a notable increase or decrease in the iteration numbers.
    
    Finally, Fig.~\ref{Fig:Pb3 only mixed} focuses solely on mixed boundary conditions. A significant change in shape and behavior is observed for two different exponential nonlinearities. Another remarkable observation lies in the even more drastic increase in the number of iterations, similar to Fig.~\ref{Fig:Pb3 no mixed} (right), that is visible in both plots in Fig.~\ref{Fig:Pb3 only mixed}. 

    In conclusion, on a fairly speculative note, our experiments may indicate that the use of mixed boundary conditions plays a more important role with regard to the iteration numbers  needed to reach the OCR by means of the Picard scheme than the strength of the (subcritical) nonlinearity.

    \begin{figure}
         \centering
         \includegraphics[scale=0.34]{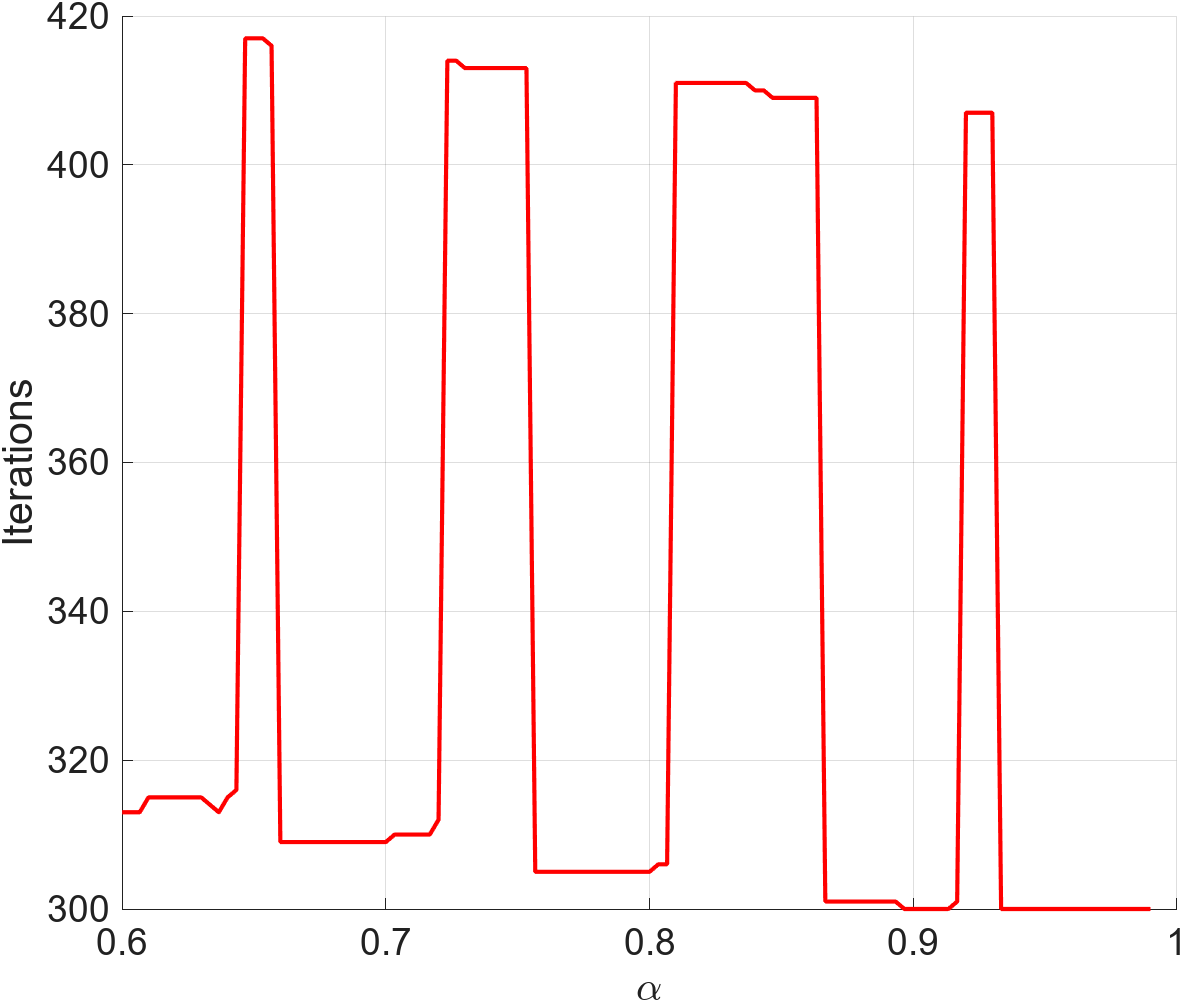}\hfill\includegraphics[scale=0.34]{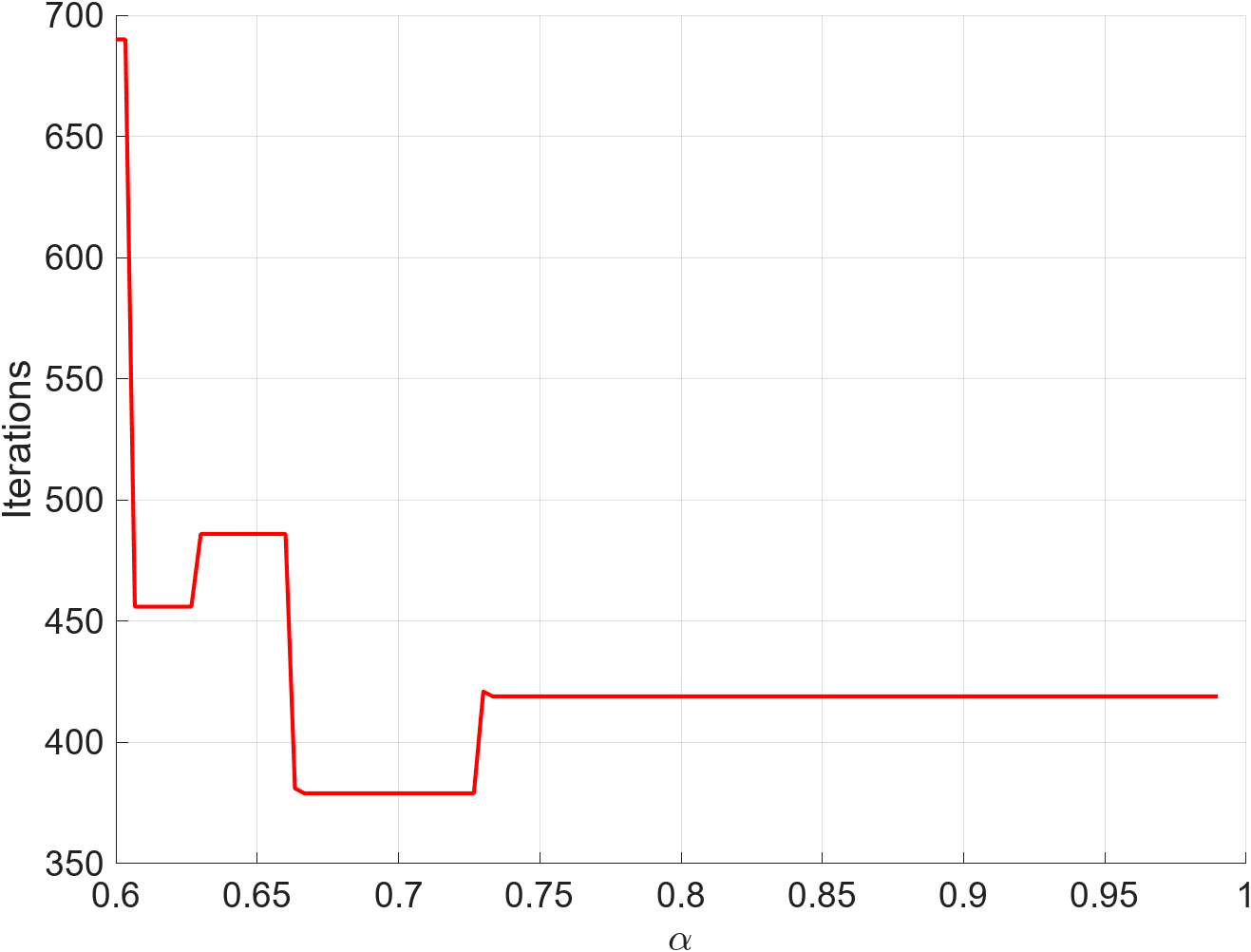}
         \caption{Numbers of iterations needed {}{across all meshes} in terms of $\alpha$ for $\gamma=10$, with mixed boundary conditions. Left: Nonlinearity $g(u)=\exp(u)$. Right: Nonlinearity $g(u)=\exp(4u)$.}
         \label{Fig:Pb3 only mixed}
    \end{figure}

\section{Conclusions} \label{sc:concl}
    In our work, we have developed the a priori error analysis of the iterative linearized Galerkin (ILG) framework for semilinear elliptic problems with mixed boundary conditions and exponentially or even subcritically growing nonlinear reactions. For such problems, we have demonstrated well-posedness under suitable monotonicity assumptions. This is accomplished through a new local analysis for a contractive Picard iteration, which, on the one hand, implies existence and uniqueness of a solution (on both the continuous and discrete level), and, on the other hand, motivates an iterative finite element scheme. A key ingredient is our modification of the Trudinger inequality that allows to remove the usual norm restriction, thereby enabling a search for weak solutions within the entire space $\spcH^1(\Omega)$; in the pure Dirichlet case, based on the Moser-Trudinger inequality, our analysis extends to certain nonlinearities of critical growth. In addition, we have derived a quasi-optimality bound that, in a subsequent step, has led to  optimal convergence for the iterative numerical approximation scheme on graded meshes. Furthermore, our numerical experiments validate the optimal convergence rate and thereby underscore the effectivity of the proposed approach in handling both challenging reaction nonlinearities and possible corner singularities. 
       
\bibliographystyle{amsalpha}
\bibliography{ref}
%\nocite{lam2010nlaplacian}
%\nocite{He2024}
%\nocite{Funken2011}
%\nocite{Alberty1999}
%\nocite{Badiale2011}
%\nocite{Brunner2024}
%\nocite{fontana2018sharp}
%\nocite{doO1995}
%\nocite{Cianchi2005}

\end{document}